\documentclass[12pt,leqno]{article}
\usepackage{amsfonts}
\usepackage{amsmath, amsthm, amsfonts, amssymb, color}
\usepackage{mathrsfs}
\renewcommand{\bar}{\overline}
\renewcommand{\hat}{\widehat}
\renewcommand{\tilde}{\widetilde}

\newtheorem{thm}{Theorem}[section]

\newtheorem{lem}[thm]{Lemma}

\newtheorem{exa}[thm]{Example}
\theoremstyle{definition}

\newcommand{\scr}[1]{\mathscr #1}
\definecolor{wco}{rgb}{0.5,0.2,0.3}

\numberwithin{equation}{section} \theoremstyle{remark}
\newtheorem{rem}{Remark}[section]

\newcommand{\ua}{\uparrow}

\title{{\bf Invariant Measures for Path-Dependent Random Diffusions\footnote{Supported in
 part by    NNSFs of China (Nos. 11301030, 11431014, 11401592) and 985-project.}}
}
\author{
{\bf  Jianhai Bao$^{b)}$,  Jinghai Shao$^{a)}$, Chenggui Yuan$^{c)}$}\\
\footnotesize{$^{a)}$Center of Applied Mathematics, Tianjin
University, Tianjin 300072, China}\\
\footnotesize{$^{b)}$School of Mathematics and Statistics, Central
South
University, Changsha 410083, China}\\
\footnotesize{$^{c)}$Department of Mathematics, Swansea University,
Singleton Park, SA2 8PP, UK}\\ \footnotesize{jianhaibao13@gmail.com,
shaojh@bnu.edu.cn, C.Yuan@swansea.ac.uk}}

\begin{document}
\def\R{\mathbb R}  \def\ff{\frac}  \def\B{\mathbf
B}
\def\N{\mathbb N} \def\kk{\kappa} \def\m{{\bf m}}
\def\dd{\delta} \def\DD{\Delta} \def\vv{\varepsilon} \def\rr{\rho}
\def\<{\langle} \def\>{\rangle} \def\GG{\Gamma} \def\gg{\gamma}
  \def\nn{\nabla} \def\pp{\partial} \def\EE{\scr E}
\def\d{\text{\rm{d}}} \def\bb{\beta} \def\aa{\alpha} \def\D{\scr D}
  \def\si{\sigma} \def\ess{\text{\rm{ess}}}
\def\beg{\begin} \def\beq{\begin{equation}}  \def\F{\scr F}
\def\Ric{\text{\rm{Ric}}} \def\Hess{\text{\rm{Hess}}}
\def\e{\text{\rm{e}}} \def\ua{\underline a} \def\OO{\Omega}  \def\oo{\omega}
 \def\tt{\tilde} \def\Ric{\text{\rm{Ric}}}
\def\cut{\text{\rm{cut}}} \def\P{\mathbb P} \def\ifn{I_n(f^{\bigotimes n})}
\def\C{\scr C}      \def\aaa{\mathbf{r}}     \def\r{r}
\def\gap{\text{\rm{gap}}} \def\prr{\pi_{{\bf m},\varrho}}  \def\r{\mathbf r}
\def\Z{\mathbb Z} \def\vrr{\varrho} \def\ll{\lambda}
\def\L{\scr L}\def\Tt{\tt} \def\TT{\tt}\def\II{\mathbb I}
\def\i{{\rm in}}\def\Sect{{\rm Sect}}\def\E{\mathbb E} \def\H{\mathbb H}
\def\M{\scr M}\def\Q{\mathbb Q} \def\texto{\text{o}} \def\LL{\Lambda}
\def\Rank{{\rm Rank}} \def\B{\scr B} \def\i{{\rm i}} \def\HR{\hat{\R}^d}
\def\to{\rightarrow}\def\l{\ell}
\def\8{\infty}\def\X{\mathbb{X}}\def\3{\triangle}
\def\V{\mathbb{V}}\def\M{\mathbb{M}}\def\W{\mathbb{W}}\def\Y{\mathbb{Y}}\def\1{\lesssim}

\def\La{\Lambda}\def\S{\mathbf{S}}

\renewcommand{\bar}{\overline}
\renewcommand{\hat}{\widehat}
\renewcommand{\tilde}{\widetilde}

\maketitle

\begin{abstract}
In this work, we are concerned with existence and uniqueness of
invariant measures for  path-dependent random diffusions and their
time discretizations. The random diffusion here means a diffusion
process living in a random environment characterized by a continuous
time Markov chain. Under certain ergodic conditions, we show that
the path-dependent random diffusion enjoys a unique invariant
probability measure and converges exponentially to its equilibrium
under the Wasserstein distance. Also, we demonstrate that the time
discretization of the path-dependent random diffusion involved
admits a unique invariant probability measure and shares the
corresponding ergodic property when the stepsize is sufficiently
small. During this procedure, the difficulty arose from the
time-discretization of continuous time Markov chain has to be deal
with, for which an estimate on its exponential functional is
presented.

\end{abstract}
 \vskip 2cm

\section{Introduction and Main Results}
A random diffusion  is a  Markov process consisting of two
components $(X(t),\LL(t))$, where the first component $X(t)$ means
the underlying continuous dynamics and the second one $\LL(t)$
stands for a   jump process. Such diffusions have a wide range of
emerging and existing applications in, for instance, climate science,
material science, molecular biology,  ecosystems, econometric
modeling, and control and optimization of largescale systems; see,
e.g., \cite{BR,KMS,LMY,MT,MTY,SYZ,ZY} and references therein.
Viewing random diffusions as   a number of diffusions with random
switching, they may be seemingly not much different from their
diffusion counterpart. Nevertheless, the coexistence of continuous
dynamics and jump processes results in challenge in dealing with
random diffusions $(X(t),\LL(t))$ under consideration, even though,
in each random temporal environment, $X(t)$ is simple enough for
intuitive understanding.  \cite{LS} revealed that $X(t)$ is
exponentially stable in $p$-th moment in a random temporal
environment and  algebraically stable in $p$-th moment in the other
scenarios, whereas $X(t)$ is ultimately exponentially stable;
\cite{PP,PS} constructed several very interesting examples to show
that $(X(t),\LL(t))$ is recurrent (resp. transient) even if
 $X(t)$ is transient (resp.
recurrent) in each random temporal environment; Unlike
Ornstein-Uhlenbeck (OU) process admitting light tail, the random OU
process enjoys heavy tail property shown in \cite{B10,DY}.

Recently,  ergodicity of random diffusions with state dependent or
state independent jump rates has been investigated extensively; see,
for example,  \cite{B10,BL,CH,Shao,Sh15,SX} for the setting of state
independent jump rates, \cite{BL,CH,Sh15} for the setup of bounded
state dependent jump rates, \cite{MT,TM} for the framework of
unbounded and state dependent jump rates. So far, there are several
approaches to explore ergodicity for random diffusions; see, for
instance, \cite{B10,BL,Sh15} via probabilistic coupling argument,
\cite{CH,MT,TM} by weak Harris' theorem, \cite{Shao,Sh15} based on
the theory of M-matrix and Perron-Frobenius theorem. For the
ergodicity of random diffusions with infinite regimes, we refer to
\cite{Shao,Sh15,XZ}.

More often than not, to understand very well the behavior of
numerous real-world systems, one of the better ways is to take the
influence of past events on the current and future states of the
systems involved  into consideration. Such point of view is
especially appropriate in the study on population biology,  neural
networks, viscoelastic materials subjected to heat or mechanical
stress, and financial products, to name a few,  since predictions on
their evolution rely heavily on  the knowledge of their past; see,
for instance, \cite{AHM,CC,MS,Mao} and references therein for more
details. There is vast literature on path-dependent ordinary
differential equation, among which the monograph \cite{HL} provides
an introduction to this subject. Also, there is a sizeable
literature on path-dependent stochastic  differential equations
(SDEs); see, e.g., \cite{HMS,KW,M84,RRV} and references therein.
Concerning existence and uniqueness of invariant probability
measures for path-dependent SDEs, we refer to \cite{RRV}, where
the drift term is semi-linear, \cite{ESV} with the drift part being
superlinear growth and satisfying a dissipativity condition, and
\cite{BS} under the extended Veretennikov-Khasminski condition.


Under certain Lyapunov condition which is not related to stationary
distribution of Markov chain involved, \cite{YM,YMa} investigated
existence and uniqueness of invariant probability measures for a
class of random diffusions by exploiting the M-matrix trick, and
\cite{YZM} further discussed the same issue for a range of
path-dependent random diffusions. Recently, under  ergodic
conditions, \cite{BSY} probed deeply into existence and uniqueness
of invariant probability measures for a kind of random diffusions by
developing new analytical frameworks.

As described above, there is a natural motivation for considering
stochastic dynamical systems, where all three features (i.e. random
switching, path dependence and noise) are present. In this work, we
are interested in ergodic properties for path-dependent random
diffusions. More precisely, as a continuation of \cite{BSY}, in the
current work we are concerned with existence and uniqueness of
invariant probability measures not only for path-dependent random
diffusions but also for their time discretizations.   In comparison
with \cite{BSY,YM,YMa}, the difficulties to deal with existence and
uniqueness of (numerical) invariant probability measures for
path-dependent random diffusions lie in: (i) the state space of
functional solutions $(X_t)_{t\ge0}$ is an infinite-dimensional
space; (ii) Both components of $(X_t,\LL(t))$ are discretized,
where,  in particular, time discretization of continuous time Markov
chain causes additional difficulties in analyzing the long-term
behavior of numerical scheme; (iii) Our investigation is based on
certain ergodic conditions. So, it turns out to be much more
challenging to cope with long term (numerical) behavior of
path-dependent random diffusions. In the present work, it is worthy
to pointing out that $(X_t,\LL(t))$ possesses a unique invariant
probability measure although the functional solution $X_t$ doesn't
admit an invariant probability measure in some fixed
environment, which is quite different from the existing results;
see, e.g. \cite{BS,ESV,RRV}. For more and precise interpretations
on $(X_t)_{t\ge0}$ and $(\LL(t))_{t\ge0}$, please refer to
subsections \ref{sub1}-\ref{sub4}.

Prior to  presentation of  the setting for this work, we  consider
and introduce some notation and terminology needed in the rest of
the paper. Let $(\R^n,\<\cdot,\cdot\>,|\cdot|)$ be the
$n$-dimensional Euclidean space.
 For fixed $\tau>0$, let $\C=C([-\tau,0];\R^n)$ denote the
family of all continuous functions $f:[-\tau,0]\rightarrow\R^n$,
endowed with the uniform norm
$\|f\|_\8:=\sup_{-\tau\le\theta\le0}|f(\theta)|$. Let
$\mathbf{S}=\{1,2\cdots,N\}$ for some integer $N\in[2,\8)$. Let
$(\LL(t))$ stand  for  a continuous--time Markov chain with the
state space $\mathbf{S}$, and the transition rules
specified by
\begin{equation}\label{y1}
\P(\LL(t+\3)=j|\LL(t)=i)=
\begin{cases}
q_{ij}\3+o(\3),~~~~~~~~i\neq j\\
1+q_{ii}\3+o(\3),~~~i=j
\end{cases}
\end{equation}
provided $\3\downarrow0$, where $o(\3)$ means that
$\lim_{\3\rightarrow0}\ff{o(\3)}{\3}=0$, and $Q=(q_{ij})$ be the
$Q$-matrix associated with the Markov chain $(\LL(t))$. Let $(W(t))$
be an $m$-dimensional Brownian motion.  We assume that $(\LL(t))$
 is irreducible, together with the finiteness of $\mathbf{S}$, which yields the positive
recurrence. Let
$\pi=(\pi_1,\cdots,\pi_N)$ denote its stationary distribution, which
can be solved by $\pi Q=0$ subject to $\sum_{i\in\S}\pi_i=1$ with
$\pi_i\ge0$. Assume that $(\LL(t))$ is independent of $(W(t))$. Let
$\|\cdot\|_{\rm HS}$ means the Hilbert-Schmidt norm. Let ${\bf
E}=\C\times\S$. For any ${\bf x}=(\xi,i)\in{\bf E}$ and  ${\bf
y}=(\eta,j)\in{\bf E}$, define the distance $\rr$ between ${\bf x}$ and
${\bf y}$  by
\begin{equation*}
\rr({\bf x},{\bf y})=\|\xi-\eta\|_\infty+{\bf1}_{\{i\neq j\}},
\end{equation*}
where, for a set $A$, ${\bf1}_A(x)=1$ with $x\in A$; otherwise,
${\bf1}_A(x)=0$. Let $\mathcal {P}=\mathcal {P}({\bf E})$ be the
space of all probability measures on ${\bf E}$.
Set
\[\mathcal{P}_0=\{\nu\in \mathcal{P}; \ \int_{\mathbf{E}} \|\xi\|_{\infty}\nu(\d \xi)<\infty\}. \]
Define the
Wasserstein distance $W_\rr$ between two probability measures $\mu,\,\nu \in\mathcal {P}_0$ as follows:
\begin{equation*}
W_\rr(\mu,\nu)=\inf_{\pi\in\mathcal {C}(\mu,\nu)}\Big\{\int_{{\bf
E}\times{\bf E}}\rr({\bf x},{\bf y})\pi(\d {\bf x},\d {\bf
y})\Big\}<\infty,
\end{equation*}
where $\mathcal {C}(\mu,\nu)$ denotes the collection of all
probability measures on ${\bf E}\times{\bf E}$ with marginals $\mu$
and $\nu$, respectively.  In this work,
$c>0$ will stand for a generic constant which might change from
occurrence to occurrence.

Next, we present the framework of this work and state our main
results.

\subsection{Invariant Measures:  Additive Noises}\label{sub1}
In this subsection, we focus on a  path-dependent random diffusion
with additive noise
\begin{equation}\label{h2}
\d X(t)=b(X_t,\LL(t))\d t+\si(\LL(t))\d
W(t),~~~t>0,~~~X_0=\xi\in\C,~~~\LL(0)=i_0\in\S,
\end{equation}
where  $b:\C\times\S\rightarrow\R^n$,
$\si:\S\rightarrow\R^n\otimes\R^m$, and, for fixed $t\ge0,$
$X_t(\theta)=X(t+\theta)$, $\theta\in[-\tau,0],$ used the standard
notation.

We assume that, for each $i\in\S$ and arbitrary $\xi,\eta\in\C,$
\begin{enumerate}
\item[({\bf A})]
There exist $\aa_i\in\R$ and $\bb_i\in\R_+$ such that
\begin{equation*}
2\<\xi(0)-\eta(0),b(\xi,i)-b(\eta,i)\> \le
\aa_i|\xi(0)-\eta(0)|^2+\bb_i\|\xi-\eta\|_\8^2.
\end{equation*}

\end{enumerate}

Under ({\bf A}), in terms of  \cite[Theorem 2.3]{VS},  \eqref{h2}
admits a unique strong solution $(X(t;\xi,i_0))$ with the initial
datum $X_0=\xi\in\C$  and $\LL(0)=i_0\in\S$. The  segment process
(i.e., functional solution) associated with the solution process
$(X(t;\xi,i_0))$  is denoted by $(X_t(\xi,i_0))$. The pair
$(X_t(\xi,i_0),\LL(t))$ is a homogeneous Markov process; see, for
instance, \cite[Theorem 1.1]{M84} \& \cite[Proposition 3.4]{RRV}.

For $(\aa_i)$ and $(\bb_i)$   introduced in ({\bf A}), set
\begin{equation}\label{q11}
\hat \aa:=\min_{i\in\S}\aa_i,~~ \check{\aa}:=\max_{i\in\S}|\aa_i|
~~\text{ and }~~ \check{\bb}:=\max_{i\in\S}\bb_i.\end{equation}
Moreover, set
$$ Q_1:=Q+\,\mathrm{diag}\Big(\aa_1+\e^{-\hat{\alpha}\tau}\bb_1,
\cdots,\aa_N+\e^{-\hat{\alpha}\tau}\bb_N\Big), $$ where $Q$ is the
$Q$-matrix of the Markov chain $(\LL(t))$, $\tau>0$ is the length
of time lag, and $\mathrm{diag}(x_1,\ldots,x_N)$ denotes the diagonal matrix generated by the vector $(x_1,\ldots,x_N)$. Let \begin{equation}\label{b3} \eta_1=
-\max_{\gamma\in\mathrm{spec}(Q_1)}\mathrm{Re}(\gamma),
\end{equation}
where $\mathrm{spec}(Q_1)$ and $\mathrm{Re}(\gg)$ denote respectively the
spectrum (i.e., the multiset of its eigenvalues) of $Q_1$  and the
real part of $\gg$.
Let $(\La^i(t),\La^j(t))$ be the independent coupling of the $Q$-process $(\La(t))$ with starting point $(\La^i(0),\La^{j}(0))=(i,j)$.
Let $T=\inf\{t\ge0:\LL^i(t)=\LL^j(t)\}$ be  the coupling time of
$(\LL^i(t),\LL^j(t))$.
Since the cardinality of $\S$ is finite  and $(q_{ij})$ is
irreducible, there exists a constant $\theta>0$ such that
\begin{equation}\label{b22}
\P(T>t)\le\e^{-\theta t},~~~~t>0.
\end{equation}
Let $P_t((\xi,i),\cdot) $ be the transition kernel of
$(X_t(\xi,i),\LL^i(t))$. For $\nu\in\mathcal {P}$, $\nu P_t$ denotes
the law of $(X_t(\xi,i),\LL^i(t))$ when $(X_0(\xi,i),\LL^i(0))$ is
distributed according to $\nu\in\mathcal {P}.$

Our first main result  in this paper is stated as follows.
\begin{thm}\label{th4}
{\rm Suppose ({\bf A}) holds and  $\eta_1>0$. Then, it holds that 
\begin{equation}\label{hui72}
W_\rr(\nu_1 P_t,\nu_2  P_t)\le
c\Big(1+\sum_{i\in\S}\int_\C\|\xi\|_\8\nu_1(\d\xi,i)+\sum_{i\in\S}\int_\C\|\eta\|_\8\nu_2(\d\eta,i)\Big)\e^{-\ff{\theta\eta_1}{2(\theta+\eta_1)}t}
\end{equation}
for any $\nu_1,\nu_2\in\mathcal {P}_0$, where $\eta_1$ is defined in
\eqref{b3} and $\theta>0$ is specified in \eqref{b22}.
 Furthermore, \eqref{hui72} implies that $(X_t(\xi,i),\LL^i(t))$,
 determined by \eqref{h2} and \eqref{y1},
admits a unique invariant probability measure $\mu\in\mathcal {P}_0$
such that
\begin{equation}\label{t1}
W_\rr(\dd_{(\xi,i)}P_t,\mu P_t)\le
c\Big(1+\|\xi\|_\8+\sum_{i\in\S}\int_\C\|\eta\|_\8\mu(\d\eta,i)\Big)\e^{-\ff{\theta\eta_1}{2(\theta+\eta_1)}t},
\end{equation}
where $\dd_{(\xi,i)}$ stands for the Dirac's measure at the point
$(\xi,i)$. }
\end{thm}

\begin{rem}\label{re1}
{\rm  If the assumption $\eta_1>0$ is replaced by
\begin{equation*}
\sum_{i\in\S}(\aa_i+\e^{-\hat{\aa}\,\tau}\bb_i)\pi_i<0
\end{equation*}
and \begin{equation*}
\min_{i\in\mathbb{S},\aa_i+\e^{-\hat{\aa}\,\tau}\bb_i>0}\Big(-\ff{q_{ii}}{\aa_i+\e^{-\hat{\aa}\,\tau}\bb_i}\Big)>1,
\end{equation*}
according to \cite[Propositions 4.1 \& 4.2]{B10}, Theorem \ref{th4}
still holds true. }
\end{rem}

\begin{rem}
{\rm From Theorem \ref{th4} and Remark \ref{re1}, $(X_t,\LL(t))$
might have   a unique invariant probability measure even though the
functional solution $X_t$ does not admit an invariant probability
measure in a random temporal environment just as Example \ref{exa1}
below shows.

}
\end{rem}

\subsection{Invariant Measures:  Multiplicative Noises}\label{sub2}
In this subsection, we move on to consider existence and uniqueness
of  invariant probability measures  under a little bit strong
assumptions but for path-dependent random diffusions  with
multiplicative noises in the form
\begin{equation}\label{eq1}
\d X(t)=b(X_t,\LL(t))\d t+\si(X_t,\LL(t))\d W(t),~~~t>0,
~~X_0=\xi,~~\LL(0)=i_0\in\S,
\end{equation}
where $b:\C\times\S\rightarrow\R^n$ and
$\si:\C\times\S\rightarrow\R^n\otimes\R^m.$

Let $v(\cdot)$ be a probability measure on $[-\tau,0]$ and suppose
that, for any $\xi,\eta\in\C$ and each $i\in\S,$
\begin{enumerate}
\item[({\bf H1})]
There exist $\aa_i\in\R$ and $\bb_i\in\R_+$ such  that
\begin{equation*}
\begin{split}
2&\<\xi(0)-\eta(0),b(\xi,i)-b(\eta,i)\>+\|\si(\xi,i)-\si(\eta,i)\|_{\rm HS}^2\\
&\le
\aa_i|\xi(0)-\eta(0)|^2+\bb_i\int_{-\tau}^0|\xi(\theta)-\eta(\theta)|^2v(\d\theta).
\end{split}
\end{equation*}
\item[({\bf H2})] There exists an $L>0$ such that
\begin{equation*}
\|\si(\xi,i)-\si(\eta,i)\|^2_{\rm HS}\le
L\Big(|\xi(0)-\eta(0)|^2+\int_{-\tau}^0|\xi(\theta)-\eta(\theta)|^2v(\d\theta)\Big).
\end{equation*}
\end{enumerate}

For $(\aa_i)$ and $(\bb_i)$ stipulated in (\textbf{H1}), we set
\begin{equation*}
Q_2:=Q+\,\mbox{diag}\Big(\aa_1+\bb_1\int_{-\tau}^0\e^{\hat{\aa}\theta}v(\d\theta),
\cdots,\aa_N+\bb_N\int_{-\tau}^0\e^{\hat{\aa}\theta}v(\d\theta)\Big),
\end{equation*}
where $\hat\aa$ is defined as in \eqref{q11}. Furthermore, we define
 \begin{equation}\label{q13}
 \eta_2=
-\max_{\gamma\in\mathrm{spec}(Q_2)}\mathrm{Re}(\gamma).
\end{equation}

Under appropriate assumptions, the semigroup generated by the pair
$(X_t(\xi,i),\LL^i(t))$ converges exponentially to the equilibrium
under the Wasserstein distance as one of the main results below
reads.

\begin{thm}\label{th0}
{\rm Let ({\bf H1})-({\bf H2}) hold and assume further $\eta_2>0$.
Then,
\begin{equation}\label{hui7}
W_\rr(\nu_1 P_t,\nu_2 P_t)\le
c\,\Big(1+\sum_{i\in\S}\int_\C\|\xi\|_\8\nu_1(\d\xi,i)+\sum_{i\in\S}\int_\C\|\eta\|_\8\nu_2(\d\eta,i)\Big)\e^{-\ff{\theta\eta_2}{2(\theta+\eta_2)}t}
\end{equation}
for any $\nu_1,\nu_2\in\mathcal {P}_0$, where $\theta>0$ such that
\eqref{b22} holds and $\eta_2>0$ is defined in \eqref{q13}.
Furthermore, \eqref{hui7} implies that $(X_t(\xi,i),\LL^i(t))$
solving \eqref{eq1} and \eqref{y1} and admits a unique invariant
probability measure $\mu\in\mathcal {P}_0$ such that
\begin{equation}\label{eq12}
W_\rr(\dd_{(\xi,i)}P_t,\mu P_t)\le
c\Big(1+\|\xi\|_\8+\sum_{i\in\S}\int_\C\|\eta\|_\8\mu(\d\eta,i)\Big)\e^{-\ff{\theta\eta_2}{2(\theta+\eta_2)}t}.
\end{equation}
}
\end{thm}

Next, we provide an example to demonstrate Theorem \ref{th0}.

\begin{exa}\label{exa1}
{\rm Let $(\LL(t))_{t\ge0}$ be a Markov chain taking values in
$\mathbb{S}=\{1,2\}$ with the generator
\begin{equation}\label{T2}
Q= \left(\begin{array}{ccc}
  -1 & 1\\
  \gamma & -\gamma\\
  \end{array}
  \right)
  \end{equation}
for some constant $\gamma>0.$ Consider a scalar path-dependent OU
 process
\begin{equation}\label{eq11}
\d X(t)=\{a_{\LL(t)}X(t)+b_{\LL(t)}X(t-1)\}\d t+\si_{\LL_t}\d
W(t),~t>0,~(X_0,\LL(0))=(\xi,1)\in\C \times\mathbb{S},
\end{equation}
where
 $a_1,b_1,b_2>0,a_2<0$. Set $\aa:=2a_1+(1+\e^{-a_2})b_1$,
 $\bb:=2a_2+(1+\e^{-a_2})b_2$. For $\aa,\gg>0, \bb\in\R$ above, if
\begin{equation}\label{hui6}
\begin{cases}
\aa+\bb<1+\gg\\
\bb-\ff{\bb}{\aa}>\gg.
\end{cases}
\end{equation}
then $(X_t(\xi,i),\LL^i(t))$, determined by \eqref{eq11} and
\eqref{T2}, has a unique invariant probability measure, and
converges exponentially to the equilibrium.


}
\end{exa}

\subsection{Numerical Invariant Measures:  Additive
Noises}\label{sub3} In this subsection, we proceed to discuss
existence and uniqueness of invariant probability measures for the
time discretization of $(X_t(\xi,i_0),\LL^i(t))$, determined by
\eqref{h2} and \eqref{y1}, respectively, and investigate the
exponential ergodicity under the Wasserstein distance.

Without loss of generality, we assume the step size
$\dd=\ff{\tau}{M}\in(0,1)$ for some integer $M>\tau$. Consider the
following   EM scheme associated with \eqref{h2}
\begin{equation}\label{b0}
\d Y(t)=b(Y_{t_\dd},\LL({t_\dd}))\d t+\si(\LL({t_\dd}))\d
W(t),~~~~t>0
\end{equation}
with the initial condition $Y(\theta)=\xi(\theta)$ for
$\theta\in[-\tau,0]$ and $\LL(0)=i_0\in\S$, where,  $t_\dd:=\lfloor
t/\dd\rfloor\dd$ with $\lfloor t/\dd\rfloor$ being the integer part
of $t/\dd$, and $Y_{k\dd}=\{ Y_{k\dd}(\theta):-\tau\le\theta\le0\}$
is a $\C$-valued random variable defined as follows: for any
$\theta\in[i\dd,(i+1)\dd]$, $i=-M,-(M-1),\cdots,-1$,
\begin{equation}\label{w2}
Y_{k\dd}(\theta)=Y((k+i)\dd)+\ff{\theta-i\dd}{\dd}\{ Y((k+i+1)\dd)-
Y((k+i)\dd)\},~~~
\end{equation}
i.e., $Y_{k\dd}(\cdot)$ is the linear interpolation of
$Y((k-M)\dd)$, $Y((k-(M-1))\dd),\cdots,Y((k-1)\dd),Y(k\dd)$. Keep in
mind that the $\C$-valued random variables $X_t$ and $Y_{t_\dd}$ in
\eqref{eq1} and \eqref{b0}, respectively, are defined in a quite
different way. In order to emphasize the initial condition
$Y(\theta)=\xi(\theta)$ for $\theta\in[-\tau,0]$ and
$\LL(0)=i\in\S$, in some of the occasion, we shall write
$Y(t;\xi,i)$ and $Y_{t_\dd}(\xi,i)$ in lieu of $Y(t)$ and $Y_{t_\dd}
$, respectively. For latter purpose, we extend the initial value
$Y(\theta)=\xi(\theta), \theta\in[-\tau,0]$, of \eqref{b0} into the
interval $[-\tau-1,-\tau)$ by setting $Y(\theta)=\xi(-\tau)$ for any
$ \theta\in[-\tau-1,-\tau).$ Moreover, the pair
$(Y_{t_\dd}(\xi,i),\LL(t_\dd))$ enjoys the Markov property as Lemma
\ref{Markov} below shows. Let $P_{k\dd}^{(\dd)}((\xi,i),\cdot)$
stand for the transition kernel of $(Y_{k\dd}(\xi,i),\LL^i(k\dd)).$

To investigate the long-term behavior of $Y_{k\dd}$ defined by
\eqref{w2}, besides ({\bf A}), we further assume that there exists
an $L_0>0$ such that
\begin{equation}\label{w3}
|b(\xi,i)-b(\eta,i)|\le
L_0\|\xi-\eta\|_\8,~~~~~\xi,\eta\in\C,~~i\in\S.
\end{equation}

The theorem below shows that the discrete-time semigroup generated
by the discretization of $(X_t(\xi,i),\LL^i(t))$ admits a unique
invariant probability measure and is exponentially convergent to its
equilibrium under the Wasserstein distance.

\begin{thm}\label{th3}
{\rm Let the assumptions of Theorem \ref{th4} be satisfied and
suppose further \eqref{w3} holds. Then, there exist  $\dd_0\in(0,1)$
and $\aa>0$ such that for any $k\ge0$ and $\dd\in(0,\dd_0)$,
\begin{equation}\label{w9}
W_\rr(\nu_1 P_{k\dd}^{(\dd)},\nu_2  P_{k\dd}^{(\dd)})\le
c\,\Big(1+\sum_{i\in\S}\int_\C\|\xi\|_\8\nu_1(\d\xi,i)+\sum_{i\in\S}\int_\C\|\eta\|_\8\nu_2(\d\eta,i)\Big)\e^{-\aa
k\dd},
\end{equation}
in which  $\nu_1,\nu_2\in\mathcal {P}_0$. Furthermore, \eqref{w9} implies
that $(Y_{k\dd}(\xi,i),\LL^i(k\dd))$, ascertained by \eqref{b0} and
\eqref{y1}, admits a unique invariant probability measure
$\mu^{(\dd)}\in\mathcal {P}_0$ such that
\begin{equation*}
W_\rr(\dd_{(\xi,i)}P_{k\dd}^{(\dd)},\mu^{(\dd)} P_{k\dd}^{(\dd)})\le
c\,\Big(1+\|\xi\|_\8+\sum_{i\in\S}\int_\C\|\eta\|_\8\mu^{(\dd)}(\d\eta,i)\Big)\e^{-\aa
k\dd}.
\end{equation*}
}
\end{thm}

\begin{rem}
{\rm By following the argument of \cite[Theorem 3.2]{BSY}, it
follows that
\begin{equation*}
\lim_{\dd\rightarrow0}W_\rr(\mu^{(\dd)},\mu)=0,
\end{equation*}
where $\mu\in\mathcal {P}_0$ is the invariant probability measure of
$(X_t(\xi,i),\LL^i(t))$,
 determined by \eqref{h2} and \eqref{y1}, and $\mu^{(\dd)}\in\mathcal
 {P}_0$ is the invariant probability measure of $(Y_{k\dd}(\xi,i),\LL^i(k\dd))$  solving   \eqref{b0} and
\eqref{y1}. }
\end{rem}

\subsection{Numerical Invariant Measures: Multiplicative
Noises}\label{sub4} In this subsection, we move forward to discuss
the multiplicative noise case. For this setting, we further assume
that there exists an $L_1>0$ such that
\begin{equation}\label{a7}
|b(\xi,i)-b(\eta,i)|^2\le
L_1\Big(|\xi(0)-\eta(0)|^2+\int_{-\tau}^0|\xi(\theta)-\eta(\theta)|^2v(\d\theta)\Big)
\end{equation}
for any $\xi,\eta\in\C$ and $i\in\S.$ Consider the EM scheme
corresponding to \eqref{eq1}
\begin{equation}\label{a1}
\d Y(t)=b(Y_{t_\dd},\LL({t_\dd}))\d t+\si(Y_{t_\dd},\LL({t_\dd}))\d
W(t),~~~~t>0,
\end{equation}
with the initial condition $Y(\theta)=\xi(\theta)$ for
$\theta\in[-\tau,0]$ and $\LL(0)=i_0\in\S$, where $Y_{t_\dd}$ is
defined exactly as in \eqref{w2}. Set
\[
Q_3:=Q+\mathrm{diag}\Big(\alpha_1+4\e^{-\hat{\alpha}\tau}\beta_1,\ldots,
\alpha_N+4\e^{-\hat{\alpha}\tau}\beta_N\Big),
\]and
\begin{equation}
\eta_3:=-\max_{\gamma\in\mathrm{spec}(Q_3)}\mathrm{Re}(\gamma).
\end{equation}

Concerning  the multiplicative noise case, the time discretization
of $(X_t(\xi,i),\LL^i(t))$, determined by \eqref{eq1} and
\eqref{y1}, also shares the exponentially ergodic property when the
stepsize is sufficiently small, which is presented below as
another main result of this paper.

\begin{thm}\label{th5}
{\rm Let ({\bf H1}), ({\bf H2}), and \eqref{a7} hold and assume
further $\eta_3>0$. Then, there exist  $\dd_0\in(0,1)$ and $\aa>0$
such that, for any $k\ge0$ and $\dd\in(0,\dd_0),$
\begin{equation}\label{w99}
W_\rr(\nu_1 P_{k\dd},\nu_2  P_{k\dd})\le
c\,\Big(1+\sum_{i\in\S}\int_\C\|\xi\|_\8\nu_1(\d\xi,i)+\sum_{i\in\S}\int_\C\|\eta\|_\8\nu_2(\d\eta,i)\Big)\e^{-\aa
k\dd},
\end{equation}
where $\nu_1,\nu_2\in\mathcal {P}_0$. Furthermore, \eqref{w99} implies
that $(Y_{k\dd}(\xi,i),\LL^i(k\dd))$, determined by \eqref{a1} and
\eqref{y1}, admits a unique invariant probability measure
$\mu^{(\dd)}\in\mathcal {P}_0$ such that
\begin{equation*}
W_\rr(\dd_{(\xi,i)}P_{k\dd},\mu^{(\dd)} P_{k\dd})\le
c\,\Big(1+\|\xi\|_\8+\sum_{i\in\S}\int_\C\|\eta\|_\8\mu^{(\dd)}(\d\eta,i)\Big)\e^{-\aa
k\dd}.
\end{equation*}
}
\end{thm}

The remainder of this paper is organized as follows. Section
\ref{sec2} is devoted to the proof  of Theorem \ref{th4}; Section
\ref{sec3} is concerned with the proofs of Theorem \ref{th0} and
Example \ref{exa1}; In Section \ref{sec4}, we aim to investigate the
estimate on exponential functional of the discrete observation for
the Markov chain involved and meanwhile finish the proof of Theorem
\ref{th3}; At length, we focus on the Markov property of time
discretization of $(X_t(\xi,i),\LL^i(t))$  and complete the proof of
Theorem \ref{th5}.

\section{Proof of Theorem \ref{th4}}\label{sec2}

%
%
%
%


Let
$$\OO_1=\{\omega|\ \omega:[0,\8)\rightarrow\R^m \mbox{ is continuous with
} \omega(0)=0\},$$ which is endowed with the locally uniform
convergence topology and the Wiener measure $\P_1$ so that the
coordinate process $W(t,\omega):=\omega(t)$, $t\ge 0$, is a  standard
$m$-dimensional Brownian motion. Set
\begin{equation*}
\OO_2:=\Big\{\omega\big|\ \omega: [0,\infty)\rightarrow \mathbf{S}  \
\text{is right continuous with left limit}\Big\},
\end{equation*}
endowed with Skorokhod topology and a probability
measure $\P_2$ so that the coordinate process $\Lambda(t,\omega)= \omega(t)$, $t\geq 0$, is a continuous time Markov chain with $Q$-matrix $(q_{ij})$.  Let
\begin{equation*}
(\OO,\F,\P)=(\OO_1\times\OO_2,\B(\OO_1)\times\B(\OO_2),\P_1\times\P_2).
\end{equation*}
Then, under $\P:=\!\P_1\!\times\! \P_2$, for
$\omega=(\omega_1,\omega_2)\in\OO$, $\omega_1(\cdot)$ is a Brownian
motion, and $\omega_2(\cdot)$ is a continuous time Markov chain with $Q$-matrix $(q_{ij})$ on $\mathbf{S}$.  Throughout this paper, we shall work on
the probability space $(\OO,\F,\P)$ constructed above.

The lemma below shows that, under suitable assumptions, the
functional solutions starting from different points will close in
the $L^2$-norm sense to each other when  time parameter goes to
infinity.
\begin{lem}\label{lem3.1}
{\rm Under the assumptions of Theorem \ref{th4},
\begin{equation}\label{b5}
\E\|X_t(\xi,i)-X_t(\eta,i)\|_\8^2\le c\,\|\xi-\eta\|^2_\8\e^{-\eta_1
t}
\end{equation}
for any $\xi,\eta\in\C$ and $i\in\S$,
 where $\eta_1>0$ is defined in \eqref{b3}.}
\end{lem}

\begin{proof}

For fixed $\omega_2\in\OO_2$,  consider the following SDE
\begin{equation*}
\d X^{\omega_2}(t) =b(X_t^{\omega_2},\LL^{\omega_2}(t))\d
t+\si(\LL^{\omega_2}(t))\d
\omega_1(t),~~t>0,~~X_0^{\omega_2}=\xi\in\C,~~\LL^{\omega_2}(0)=i\in\S.
\end{equation*}
 Since $(\LL^{\omega_2}(s))_{s\in[0,t]}$
may own finite number of jumps,
$t\mapsto\int_0^t\aa_{\LL^{\omega_2}(s)}\d s$ need not to be
differentiable. To overcome this drawback, let us introduce a smooth
approximation of it. For any $\vv\in(0,1)$, set
\begin{equation*}
\aa_{\LL^{\omega_2}(t)}^\vv:=\ff{1}{\vv}\int_t^{t+\vv}\aa_{\LL^{\omega_2}(s)}\d
s+\vv t=\int_0^1\aa_{\LL^{\omega_2}(\vv s+t)}\d s+\vv t.
\end{equation*}
Plainly, $t\mapsto\aa_{\LL^{\omega_2}(t)}^\vv$ is continuous  and
$\aa_{\LL^{\omega_2}(t)}^\vv\to \aa_{\LL^{\omega_2}(t)}$ as
$\vv\downarrow 0$ due to the right continuity of the path of
$\LL^{\omega_2}(\cdot)$. As a consequence,
$t\mapsto\int_0^t\aa_{\LL^{\omega_2}(r)}^\vv\d r$  is differentiable
by the first fundamental theorem of calculus and
$\int_0^t\aa_{\LL^{\omega_2}(r)}^\vv\d
r\to\int_0^t\aa_{\LL^{\omega_2}(r)}\d r$ as $\vv\downarrow0 $
according to Lebesgue's dominated convergence theorem. Let
\begin{equation}\label{x1}
\Gamma^{\omega_2}(t)=X^{\omega_2}(t;\xi,i)-X^{\omega_2}(t;\eta,i).
\end{equation}
 Applying
It\^o's formula and taking ({\bf A}) into account ensures that
\begin{equation}\label{x3}
\begin{split}
\e^{-\int_0^t\aa_{\LL^{\omega_2}(s)}^\vv\d
s}|\Gamma^{\omega_2}(t)|^2&=|\Gamma^{\omega_2}(0)|^2+\int_0^t\e^{-\int_0^s\aa_{\LL^{\omega_2}(r)}^\vv\d
r}\Big\{-\aa_{\LL^{\omega_2}(s)}^\vv|\Gamma^{\omega_2}(s)|^2\\
&\quad+2\<\Gamma^{\omega_2}(s),b(X_s^{\omega_2}(\xi,i),\LL^{\omega_2}(s))-b(X_s^{\omega_2}(\eta,i),\LL^{\omega_2}(s))\>\Big\}\d
s\\
&\le|\Gamma^{\omega_2}(0)|^2+\Gamma_1^{\omega_2,\vv}(t)+\int_0^t\bb_{\LL^{\omega_2}(s)}\e^{-\int_0^s\aa_{\LL^{\omega_2}(r)}^\vv\d
r}\|\Gamma^{\omega_2}_s\|_\8^2\d s,
\end{split}
\end{equation}
where
\begin{equation}\label{x2}
\Gamma_1^{\omega_2,\vv}(t):=\int_0^t\e^{-\int_0^s\aa_{\LL^{\omega_2}(r)}^\vv\d
r}|\aa_{\LL^{\omega_2}(s)}-\aa_{\LL^{\omega_2}(s)}^\vv|\cdot|\Gamma^{\omega_2}(s)|^2\d
s.
\end{equation}
Due to the fact that
\begin{equation}\label{x6}
\begin{split}
\e^{-\int_0^t\aa_{\LL^{\omega_2}(s)}^\vv\d
s}\|\Gamma^{\omega_2}_t\|^2_\8 \le
\e^{-\hat{\aa}\tau}\Big\{\,\|\Gamma^{\omega_2}_0\|_\8^2+\sup_{(t-\tau)\vee0\le
s\le t}\Big(\e^{-\int_0^s\aa_{\LL^{\omega_2}(r)}^\vv\d
r}|\Gamma^{\omega_2}(s)|^2\Big)\Big\},
\end{split}
\end{equation}
where $\hat\aa$ is defined in \eqref{q11}, we therefore infer from
\eqref{x3} that
\begin{equation*}
\begin{split}
\e^{-\int_0^t\aa_{\LL^{\omega_2}(s)}^\vv\d
s}\|\Gamma^{\omega_2}_t\|^2_\8\le\e^{-\hat{\aa}\tau}\Big\{2\,\|\Gamma^{\omega_2}_0\|_\8^2+\Gamma_1^{\omega_2,\vv}(t)+\int_0^t\bb_{\LL^{\omega_2}(s)}\e^{-\int_0^s\aa_{\LL^{\omega_2}(r)}^\vv\d
r}\|\Gamma^{\omega_2}_s\|_\8^2\d s\Big\}.
\end{split}
\end{equation*}
Since
$\aa_{\LL^{\omega_2}(s)}^\vv\rightarrow\aa_{\LL^{\omega_2}(s)}$ so
that $\Gamma_1^{\omega_2,\vv}(t)\rightarrow0$ as $\vv\rightarrow0$,
by taking $\vv\downarrow0$ one has
\begin{equation*}
\e^{-\int_0^t\aa_{\LL^{\omega_2}(s)}\d
s}\|\Gamma^{\omega_2}_t\|^2_\8\le\e^{-\hat{\aa}\tau}\Big\{2\,\|\Gamma^{\omega_2}_0\|_\8^2+\int_0^t\bb_{\LL^{\omega_2}(s)}\e^{-\int_0^s\aa_{\LL^{\omega_2}(r)}\d
r}\|\Gamma^{\omega_2}_s\|_\8^2\d s\Big\}.
\end{equation*}
Thus, employing Gronwall's inequality followed by taking expectation
w.r.t. $\P$ yields that
\begin{equation*}
\begin{split}
\E\|X_t(\xi,i_0)-X(\eta,i_0)\|^2_\8\le
2\,\|\xi-\eta\|^2_\8\E\,\e^{\int_0^t(\aa_{\LL(s)}+\e^{-\hat{\alpha}\tau}
\bb_{\LL(s)})\d s}.
\end{split}
\end{equation*}
Consequently, the desired assertion follows from \cite[Proposition
4.1]{B10} at once.
\end{proof}

The following lemma reveals that the functional solution is
uniformly bounded in the $L^2$-norm sense.

\begin{lem}\label{th2}
{\rm Under the assumptions of Theorem \ref{th4},
\begin{equation}\label{b4}
 \sup_{t\ge0}\E\|X_t(\xi,i)\|_\8^2\le
 c\,(1+\|\xi\|^2_\8),~~~~~(\xi,i)\in\C\times\S.
\end{equation}
}
\end{lem}

\begin{proof}
Analogously, we define
\begin{equation*}
\bb_{\LL^{\omega_2}(t)}^\vv
=\ff{1}{\vv}\int_t^{t+\vv}\bb_{\LL^{\omega_2}(s)}\d s+\vv
t=\int_0^1\bb_{\LL^{\omega_2}(\vv s+t)}\d s+\vv t.
\end{equation*}
By virtue of  ({\bf A}), for any $\gg>0$, there is a
  $c_\gamma>0$  such that
\begin{equation}\label{eq6}
2\<\xi(0),b(\xi,i)\>+\|\si(i)\|_{\rm HS}^2\le c_\gamma+
(\aa_i+\gamma)|\xi(0)|^2+\bb_i\|\xi\|^2_\8.
\end{equation}
Employing   It\^o's formula and taking \eqref{eq6} into
consideration provides that
\begin{equation}\label{hui3}
\begin{split}
\e^{-\int_0^t(\gg+\aa_{\LL^{\omega_2}(s)}^\vv)\d
s}|X^{\omega_2}(t)|^2
&\le|\xi(0)|^2+c_\gg\int_0^t\e^{-\int_0^s(\gg+\aa_{\LL^{\omega_2}(r)}^\vv)\d
r}\d
s\\
&\quad+\int_0^t\bb_{\LL^{\omega_2}(s)}^\vv\e^{-\int_0^s(\gg+\aa_{\LL^{\omega_2}(r)}^\vv)\d
r} \|X^{\omega_2}_s\|^2_\8\d s+
\Gamma_2^{\omega_2,\vv}(t)+\Upsilon^{\omega_2,\vv}(t),
\end{split}
\end{equation}
where
\begin{equation}\label{x4}
\begin{split}
\Gamma_2^{\omega_2,\vv}(t):&=\int_0^t\e^{-\int_0^s(\gg+\aa_{\LL^{\omega_2}(r)}^\vv)\d
r}\Big\{|\aa_{\LL^{\omega_2}(s)}-\aa_{\LL^{\omega_2}(s)}^\vv|\cdot|X^{\omega_2}(s)|^2\\
&\quad+
|\bb_{\LL^{\omega_2}(s)}-\bb_{\LL^{\omega_2}(s)}^\vv|\cdot\|X^{\omega_2}_s\|^2_\8\Big\}\d
s,
\end{split}
\end{equation}
 and
\begin{equation*}
\Upsilon^{\omega_2,\vv}(t):=2\int_0^t\e^{-\int_0^s(\gg+\aa_{\LL^{\omega_2}(r)}^\vv)\d
r}\<X^{\omega_2}(s),\si(\LL^{\omega_2}(s))\d \omega_1(s)\>.
\end{equation*}
  For any $0\le s\le t$ with $t-s\leq \tau$ and
  $\kk\in(0,1)$,
exploiting BDG's inequality, we obtain that
\begin{equation}\label{b13}
\begin{split}
\E_{\P_1}\Big(\sup_{s\le r\le t}\Upsilon^{\omega_2,\vv}(r)\Big)
&=\E_{\P_1}\Upsilon^{\omega_2,\vv}(s)+\E_{\P_1}\Big(\sup_{s\le r\le
t}(\Upsilon^{\omega_2,\vv}(r)-\Upsilon^{\omega_2,\vv}(s))\Big)\\
&\le2\,\E_{\P_1}\Big(\sup_{s\le r\le
t}\Big|\int_s^r\e^{-\int_0^u(\gamma+\aa_{\LL^{\omega_2}(r)}^\vv)\d
r}\<X^{\omega_2}(u),\si(\LL^{\omega_2}(u))\d
\omega_1(u)\>\Big|\Big)\\
&\le c\,\E_{\P_1}\Big(
\int_s^t\e^{-2\int_0^u(\gamma+\aa_{\LL^{\omega_2}(r)}^\vv)\d
r}|X^{\omega_2}(u)|^2\cdot\|\si(\LL^{\omega_2}(u))\|_{\rm HS}^2\d
u\Big)^{1/2}\\
&\le c\,\e^{-\int_0^t(\gamma+\aa_{\LL^{\omega_2}(r)}^\vv)\d
r} \E_{\P_1}\Big(\|X^{\omega_2}_t\|^2_\8\int_s^t\e^{2\int_u^t(\gamma+\aa_{\LL^{\omega_2}(r)}^\vv)\d r}\d u\Big)^{1/2}\\
&\le c\,\e^{-\int_0^t(\gamma+\aa_{\LL^{\omega_2}(r)}^\vv)\d
r}\E_{\P_1}\|X^{\omega_2}_t\|_\8\\
&\le\kk\,\e^{\hat{\aa}\tau}\e^{-\int_0^t(\gamma+\aa_{\LL^{\omega_2}(r)}^\vv)\d
r}\E_{\P_1}\|X^{\omega_2}_t\|_\8^2+c\,\e^{-\int_0^t(\gamma+\aa_{\LL^{\omega_2}(r)}^\vv)\d
r}.
\end{split}
\end{equation}
This, together with   \eqref{x6} with $\Gamma^{\omega_2}$ being
replaced accordingly by $X^{\omega_2}$,     and \eqref{hui3},
 leads to
\begin{equation*}
\begin{split}
\e^{-\int_0^t(\gg+\aa_{\LL^{\omega_2}(s)}^\vv)\d
s}\E_{\P_1}\|X^{\omega_2}_t\|^2_\8
&\le\ff{\e^{-\hat{\aa}\tau}}{1-\kk}\Big\{2\,\|\xi\|^2_\8+c\int_0^t\e^{-\int_0^s(\gg+\aa_{\LL^{\omega_2}(r)}^\vv)\d
r}\d s+c\,\e^{-\int_0^t(\gg+\aa_{\LL^{\omega_2}(r)}^\vv)\d
r}\\
&\quad+\int_0^t\bb_{\LL^{\omega_2}(s)}^\vv\e^{-\int_0^s(\gg+\aa_{\LL^{\omega_2}(r)}^\vv)\d
r} \E_{\P_1}\|X^{\omega_2}_s\|^2_\8\d
s+\E_{\P_1}\Gamma_2^{\omega_2,\vv}(t)\Big\}.
\end{split}
\end{equation*}
Then, the application of Gronwall's inequality yields that
\begin{equation}\label{r3}
\begin{split}
&\e^{-\int_0^t(\gg+\aa_{\LL^{\omega_2}(s)}^\vv)\d
s}\E_{\P_1}\|X^{\omega_2}_t\|^2_\8\\
&\le
c\,\Big\{\|\xi\|^2_\8+\int_0^t\e^{-\int_0^s(\gg+\aa_{\LL^{\omega_2}(r)}^\vv)\d
r}\d s+\e^{-\int_0^t(\gg+\aa_{\LL^{\omega_2}(r)}^\vv)\d
r}+\E_{\P_1}\Gamma_2^{\omega_2,\vv}(t)\\
&\quad+\|\xi\|^2_\8\int_0^t\Gamma_3^{\omega_2,\vv}(s)\exp\Big(\int_s^t\Gamma_3^{\omega_2,\vv}(r)\d
r\Big)\d s\\
&\quad+\int_0^t\int_0^s\e^{-\int_0^u(\gg+\aa_{\LL^{\omega_2}(r)}^\vv)\d
r}\Gamma_3^{\omega_2,\vv}(s)\exp\Big(\int_s^t\Gamma_3^{\omega_2,\vv}(r)\d
r\Big)\d
u\d s\\
&\quad+\int_0^t\e^{-\int_0^s(\gg+\aa_{\LL^{\omega_2}(r)}^\vv)\d
r}\Gamma_3^{\omega_2,\vv}(s)\exp\Big(\int_s^t\Gamma_3^{\omega_2,\vv}(r)\d
r\Big)\d s\\
&\quad+\int_0^t\E_{\P_1}\Gamma_2^{\omega_2,\vv}(t)\Gamma_3^{\omega_2,\vv}(s)\exp\Big(\int_s^t\Gamma_3^{\omega_2,\vv}(r)\d
r\Big)\d s\Big\},
\end{split}
\end{equation}
where
\begin{equation*}
\Gamma_3^{\omega_2,\vv}(t)=\ff{\e^{-\hat{\aa}\tau}\bb_{\LL^{\omega_2}(t)}^\vv}{1-\kk}.
\end{equation*}
In the following, we go to estimate the terms in the right hand side of \eqref{r3}.
By integration by parts, one has
\begin{equation}\label{r1}
\begin{split}
& \|\xi\|^2_\8\int_0^t\Gamma_3^{\omega_2,\vv}(s)\exp\Big(\int_s^t\Gamma_3^{\omega_2,\vv}(r)\d
r\Big)\d s\\
&\quad+\int_0^t\int_0^s\e^{-\int_0^u(\gg+\aa_{\LL^{\omega_2}(r)}^\vv)\d
r}\Gamma_3^{\omega_2,\vv}(s)\exp\Big(\int_s^t\Gamma_3^{\omega_2,\vv}(r)\d
r\Big)\d
u\d s\\
&=\|\xi\|^2_\8\Big(\exp\Big(\int_0^t\Gamma_3^{\omega_2,\vv}(s)\d
s\Big)-1\Big)\\
&\quad+\int_0^t\e^{-\int_0^s(\gg+\aa_{\LL^{\omega_2}(r)}^\vv)\d
r}\Big(\exp\Big(\int_s^t\Gamma_3^{\omega_2,\vv}(r)\d r\Big)-1\Big)\d
s\\
&\le\|\xi\|^2_\8\exp\Big(\int_0^t\Gamma_3^{\omega_2,\vv}(s)\d
s\Big)\\
&\quad+\int_0^t\e^{-\int_0^s(\gg+\aa_{\LL^{\omega_2}(r)}^\vv)\d
r}\exp\Big(\int_s^t\Gamma_3^{\omega_2,\vv}(r)\d r\Big)\d s.
\end{split}
\end{equation}
On the other hand,
\begin{equation}\label{q3}
\begin{split}
&\int_0^t\e^{-\int_0^s(\gg+\aa_{\LL^{\omega_2}(r)}^\vv)\d
r}\Gamma_3^{\omega_2,\vv}(s)\exp\Big(\int_s^t\Gamma_3^{\omega_2,\vv}(r)\d
r\Big)\d s\\
&\le\ff{\e^{-\hat{\aa}\tau}\check\bb}{1-\kk}\int_0^t\e^{-\int_0^s(\gg+\aa_{\LL^{\omega_2}(r)}^\vv)\d
r}\exp\Big(\int_s^t\Gamma_3^{\omega_2,\vv}(r)\d r\Big)\d s.
\end{split}
\end{equation}
Under ({\bf A}), it is quite standard to show by using H\"older's
inequality and BDG's inequality that
\begin{equation*}
\E_{\P_1}\Big(\sup_{0\le s\le t}\|X^{\omega_2}_s\|^2_\8\Big)<\8.
\end{equation*}
So, the dominated convergence theorem implies that
\begin{equation}\label{r2}
\E_{\P_1}\Gamma_2^{\omega_2,\vv}(t)+\int_0^t\E_{\P_1}\Gamma_2^{\omega_2,\vv}(t)\Gamma_3^{\omega_2,\vv}(s)\exp\Big(\int_s^t\Gamma_3^{\omega_2,\vv}(r)\d
r\Big)\d s \rightarrow 0,\quad \ \text{as $\vv \downarrow 0$}.
\end{equation}
Whereafter,  taking \eqref{r1}-\eqref{r2}  into account and keeping
in mind that   $\aa_{\LL^{\omega_2}(t)}^\vv\to
\aa_{\LL^{\omega_2}(t)},~ \bb_{\LL^{\omega_2}(t)}^\vv\to
\bb_{\LL^{\omega_2}(t)}$ as $\vv\rightarrow0$, we deduce from
\eqref{r3} that
\begin{equation}\label{b1}
\begin{split}
\E\|X_t\|^2_\8
&\le c\,\|\xi\|_\8^2\,\E\exp\Big(\int_0^t\Big(\gg
+\aa_{\LL(s)}+\ff{\e^{-\hat{\aa}\tau}}{1-\kk}\bb_{\LL(s)}\Big)\d
s\Big)\\
&\quad + c\,\Big(1+ \int_0^t\E\exp\Big(\int_s^t\Big(\gg
+\aa_{\LL(u)}+\ff{\e^{-\hat{\aa}\tau}}{1-\kk}\bb_{\LL(u)}\Big)\d
u\Big)\d s\Big).
\end{split}
\end{equation}
Accordingly, as $\eta_1>0$, by \cite[Proposition 4.1]{B10}, we
obtain that for sufficiently small $\gamma,\,\kappa\in (0,1)$,
\[\sup_{t\ge0}\E\exp\Big(\int_0^t\Big(\gg
+\aa_{\LL(s)}+\ff{\e^{-\hat{\aa}\tau}}{1-\kk}\bb_{\LL(s)}\Big)\d
s\Big)<\infty,\]
\[ \sup_{t\ge0}\int_0^t\E\exp\Big(\int_s^t\Big(\gg
+\aa_{\LL(u)}+\ff{\e^{-\hat{\aa}\tau}}{1-\kk}\bb_{\LL(u)}\Big)\d
u\Big)\d s<\infty,\]
and hence \eqref{b4} holds.
\end{proof}

We now in position to complete the

\smallskip
\noindent{\bf Proof of Theorem \ref{th4}.} For $\bb\in(0,1)$ to be
determined, by H\"older's inequality, it follows that
\begin{equation*}
\begin{split}
W_\rr(\dd_{(\xi,i)}P_t,\dd_{(\eta,j)}P_t)&\le
\E\{\|X_t(\xi,i)-X_t(\eta,j)\|_\8+{\bf1}_{\{\LL^i(t)\neq\LL^j(t)\}}\}\\
&=
\E\{(\|X_t(\xi,i)-X_t(\eta,j)\|_\8+{\bf1}_{\{\LL^i(t)\neq\LL^j(t)\}}){\bf1}_{\{T\le\bb
t\}}\}\\
&\quad+\E\{(\|X_t(\xi,i)-X_t(\eta,j)\|_\8+{\bf1}_{\{\LL^i(t)\neq\LL^j(t)\}}){\bf1}_{\{T>\bb
t\}}\}\\
&\le \E({\bf1}_{\{T\le\bb
t\}}\E\{\|X_t(\xi,i)-X_t(\eta,j)\|_\8\}|\F_T)\\
&\quad+2\{1+\sqrt{2(\E\|X_t(\xi,i)\|^2_\8+\E\|X_t(\eta,j)\|_\8^2)}\}\sqrt{\P(T>\bb t)}\\
&\le c\,\E({\bf1}_{\{T\le\bb
t\}}\|X_T(\xi,i)-X_T(\eta,j)\|_\8\e^{-\ff{\eta_1}{2}
(t-T)})\\
&\quad+c\,(1+\|\xi\|_\8+\|\eta\|_\8)\e^{-\ff{1}{2}\theta\bb t}\\
&\le c\,(1+\|\xi\|_\8+\|\eta\|_\8)(\e^{-\ff{1}{2}\theta\bb
t}+\e^{-\ff{\eta_1}{2} (1-\bb)t}),
\end{split}
\end{equation*}
where in the last two steps we have used \eqref{b5} and \eqref{b4}. Optimizing over $\bb$ in
order to have $\theta\bb=\eta_1(1-\bb)$, i.e.,
$\bb=\ff{\eta_1}{\theta+\eta_1}$, leads to
\begin{equation}\label{hui0}
\begin{split}
W_\rr(\dd_{(\xi,i)}P_t,\dd_{(\eta,j)}P_t) &\le
c\,(1+\|\xi\|_\8+\|\eta\|_\8)\e^{-\ff{\theta\eta_1}{2(\theta+\eta_1)}t}.
\end{split}
\end{equation}
Thus, substituting \eqref{hui0} into
\begin{equation*}
W_\rr(\nu_1 P_t,\nu_2 P_t)\le\int
W_\rr(\dd_{(\xi,i)}P_t,\dd_{(\eta,j)}P_t)\pi(\d\xi\times\d\{i\},\d\eta\times\d\{j\})
\end{equation*}
yields the desired assertion \eqref{hui72}, where $\pi$ is a
coupling of $\nu_1$ and $\nu_2$.

Fix $\nu\in\mathcal {P}_0$ and observe  that $(\nu P_n)_{n\ge0}$ is a
Cauchy sequence under the Wasserstein distance $W_\rho$ due to \eqref{hui72} and that $\nu P_{n+1}=\nu
P_n P_1$. So, by letting $n\rightarrow\8$, there exists
$\nu_\8\in\mathcal {P}_0$ such that $\nu_\8P_1=\nu_\8$. Set
$\mu:=\int_0^1\nu_\8 P_s\d s$. It is easy to check $\mu\in \mathcal{P}_0$. In what follows, we claim  that $\mu$
 is indeed an invariant probability measure. In fact, for any $t>0$, note  that
\begin{equation*}
\begin{split}
\mu P_t&=\int_0^1\nu_\8 P_{s+t}\d
s=\int_t^{t+1}\nu_\8 P_s\d s\\
&=\int_t^0\nu_\8 P_s\d s+\int_0^1\nu_\8 P_s\d s+\int_0^t\nu_\8
P_1P_s\d s\\ &=\mu,
\end{split}
\end{equation*}
where in the last display we have used $\nu_\8 P_1=\nu_\8 $. Let
$\mu,\tilde\mu\in\mathcal {P}_0$ both be the invariant probability
measures of $(X_t(\xi,i),\LL^i(t))$. By the invariance, we deduce
from \eqref{hui72} that
\begin{equation}\label{eq15}
W_\rho(\mu,\tilde\mu)=W_\rho(\mu P_t,\tilde\mu P_t)\le
c\,\Big(1+\sum_{i\in\S}\int_\C\|\xi\|_\8\mu(\d\xi,i)+\sum_{i\in\S}\int_\C\|\eta\|_\8\tilde\mu(\d\eta,i)\Big)\e^{-\ff{\theta\eta_1}{2(\theta+\eta_1)}t}.
\end{equation}
Consequently, the uniqueness of invariant measure can be obtained
since the right hand side of \eqref{eq15} tends to zero as $t$ goes
to infinity. Finally, \eqref{t1} follows by just taking
$\nu_1=\dd_{(\xi,i)}$  and $\nu_2=\mu$ in \eqref{hui72}.

\section{Proof of Theorem \ref{th0}}\label{sec3}

With the aid of Lemmas \ref{th1} and \ref{le1} below, the argument
of Theorem \ref{th0} can be completed by repeating the procedure of
Theorem \ref{th4}.


\begin{lem}\label{th1}
{\rm Under the assumptions of Theorem \ref{th0}, it holds that
\begin{equation}\label{eq10}
\E\|X_t(\xi,i)-X_t(\eta,i)\|_\8^2\le c\,\|\xi-\eta\|^2_\8\e^{-\eta_2
t}
\end{equation}
for any $\xi,\eta\in\C$ and $i\in\S$, where $\eta_2>0$ is defined in
\eqref{q13}.

}
\end{lem}

\begin{proof}
Fix $\omega_2\in\OO_2$ and let $(X^{\omega_2}(t))$ solve the SDE
\begin{equation*}
\d X^{\omega_2}(t) =b(X_t^{\omega_2},\LL^{\omega_2}(t))\d
t+\si(X^{\omega_2}_t,\LL^{\omega_2}(t))\d \omega_1(t), ~t>0,~
X_0^{\omega_2}=\xi\in\C,~  \LL^{\omega_2}(0)=i\in\S.
\end{equation*}
Let $\Gamma^{\omega_2}(t)$ and $\Gamma_1^{\omega_2,\vv}(t)$ be
defined as in \eqref{x1} and \eqref{x2}, respectively. By the It\^o
formula, we deduce from ({\bf H1}) that
\begin{equation}\label{c2}
\begin{split}
&\e^{-\int_0^t\aa_{\LL^{\omega_2}(s)}^\vv\d
s}\E_{\P_1}|\Gamma^{\omega_2}(t)|^2\\&=|\Gamma^{\omega_2}(0)|^2+\int_0^t\e^{-\int_0^s\aa_{\LL^{\omega_2}(r)}^\vv\d
r}\E_{\P_1}\Big\{-\aa_{\LL^{\omega_2}(s)}^\vv|\Gamma^{\omega_2}(s)|^2\\
&\quad+2\<\Gamma^{\omega_2}(s),b(X_s^{\omega_2}(\xi,i),\LL^{\omega_2}(s))-b(X_s^{\omega_2}(\eta,i),\LL^{\omega_2}(s))\>\\
&\quad+\|\si(X_s^{\omega_2}(\xi,i),\LL^{\omega_2}(s))-\si(X_s^{\omega_2}(\eta,i),\LL^{\omega_2}(s))\|^2_{\rm
HS}\Big\}\d
s\\
&\le|\Gamma^{\omega_2}(0)|^2+\E_{\P_1}\Gamma_1^{\omega_2,\vv}(t)+\int_0^t\bb_{\LL^{\omega_2}(s)}\e^{-\int_0^s\aa_{\LL^{\omega_2}(r)}^\vv\d
r}
\int_{-\tau}^0\E_{\P_1}|\Gamma^{\omega_2}(s+\theta)|^2v(\d\theta)\d
s\\
&\le
 c\,\|\Gamma^{\omega_2}_0\|_\8^2+\E_{\P_1}\Gamma_1^{\omega_2,\vv}(t)+\int_0^t\Gamma_4^{\omega_2,\vv}(s)\e^{-\int_0^s\aa_{\LL^{\omega_2}(r)}^\vv\d
r}\E_{\P_1}|\Gamma^{\omega_2}(s)|^2\d s,
\end{split}
\end{equation}
where in the last step we have   used the fact that
\begin{equation*}
\begin{split}
&\int_0^t\bb_{\LL^{\omega_2}(s)}\e^{-\int_0^s\aa_{\LL^{\omega_2}(r)}^\vv\d
r}
\int_{-\tau}^0\E_{\P_1}|\Gamma^{\omega_2}(s+\theta)|^2v(\d\theta)\d
s\\
&=\int_{-\tau}^0\int_{\theta}^{t+\theta}\bb_{\LL^{\omega_2}(s-\theta)}\e^{-\int_0^{s-\theta}\aa_{\LL^{\omega_2}(r)}^\vv\d
r} \E_{\P_1}|\Gamma^{\omega_2}(s)|^2\d sv(\d\theta)\\
&\le
c\,\|\Gamma^{\omega_2}_0\|_\8^2+\int_0^t\Gamma_4^{\omega_2,\vv}(s)\e^{-\int_0^s\aa_{\LL^{\omega_2}(r)}^\vv\d
r}\E_{\P_1}|\Gamma^{\omega_2}(s)|^2\d s
\end{split}
\end{equation*}
with
\begin{equation*}
\begin{split}
\Gamma^{\omega_2,\vv}_4(t):=\int_{-\tau}^0\bb_{\LL^{\omega_2}(t-\theta)}\e^{-\int_t^{t-\theta}\aa_{\LL^{\omega_2}(r)}^\vv\d
r}v(\d\theta).
\end{split}
\end{equation*}
Then, applying Gronwall's inequality yields that
\begin{equation}\label{b18}
\E_{\P_1}|\Gamma^{\omega_2}(t)|^2\le\{c\,\|\Gamma^{\omega_2}_0\|_\8^2+\E_{\P_1}\Gamma_1^{\omega_2,\vv}(t)\}\e^{\int_0^t(\aa_{\LL^{\omega_2}(s)}^\vv+\Gamma_4^{\omega_2,\vv}(s))\d
s}.
\end{equation}
Letting $\vv\rightarrow0$ followed by taking expectation w.r.t.
$\P_2$ on both sides of \eqref{b18}, together with
$\int_0^t\aa_{\LL^{\omega_2}(r)}^\vv\d
r\to\int_0^t\aa_{\LL^{\omega_2}(r)}\d r$ and
$\E_{\P_1}\Gamma_1^{\omega_2,\vv}(t)\rightarrow0$ as $\vv\downarrow0
$, gives that
\begin{equation}\label{v2}
\E|\Gamma(t)|^2\le
c\,\|\Gamma_0\|_\8^2\E\,\exp\Big(\int_0^t\Big(\aa_{\LL(s)}+\int_{-\tau}^0\bb_{\LL(s-\theta)}\e^{-\int_s^{s-\theta}\aa_{\LL(r)}\d
r}v(\d\theta)\Big)\d s\Big),
\end{equation}
where $\Gamma(t):=X(t;\xi,i)-X(t;\eta,i).$ It is readily to see that
\begin{align*}
\int_0^t\int_{-\tau}^0\bb_{\LL(s-\theta)}\e^{-\int_s^{s-\theta}\aa_{\LL(r)}\d
r}v(\d\theta)\d s&\le\int_{-\tau}^0\e^{\hat{\aa}\theta}
\int_{-\theta}^{t-\theta}\bb_{\LL(s)}\d sv(\d\theta)\\
&\le c+\int_{-\tau}^0\e^{\hat{\aa}\theta}v(\d\theta)
\int_0^t\bb_{\LL(s)}\d s.
\end{align*}
Inserting this into \eqref{v2}, one has
\begin{equation*}
\E|\Gamma(t)|^2\le
c\,\|\Gamma_0\|_\8^2\E\exp\Big(\int_0^t\Big(\aa_{\LL(s)}+\int_{-\tau}^0\e^{\hat{\aa}\theta}v(\d\theta)
\bb_{\LL(s)}\Big)\d s\Big).
\end{equation*}
According to \cite[Proposition 4.1]{B10}, we derive from  $\eta_2>0$
that
\begin{equation}\label{v1}
\E|\Gamma(t)|^2\le c\,\e^{-\eta_2t}\,\|\Gamma_0\|_\8^2.
\end{equation}
Next, for any $0\le s\le t,$ applying It\^o's formula and BDG's
inequality and making advantage of
 ({\bf H1}) and ({\bf H2}), we find that
\begin{equation*}
\begin{split}
\E\Big(\sup_{s\le r\le t}|\Gamma(r)|^2\Big) &\le
\E|\Gamma(s)|^2+(\check\aa+\check\bb)\int_{s-\tau}^t\E|\Gamma(r)|^2
\d r\\
&\quad+8\sqrt{2}\Big(\E\Big(\int_s^t|\Gamma(r)|^2\cdot\|\si(X_r(\xi,i),\LL(r))-\si(X_r(\eta,i),\LL(r))\|_{\rm HS}^2\d r\Big)^{1/2}\Big)\\
&\le \E|\Gamma(s)|^2+c\int_{s-\tau}^t\E|\Gamma(r)|^2 \d
r+\ff{1}{2}\E\Big(\sup_{s\le r\le t}|\Gamma(r)|^2\Big),
\end{split}
\end{equation*}
which further implies that
\begin{equation}\label{v7}
\begin{split}
\E\Big(\sup_{s\le r\le t}|\Gamma(r)|^2\Big) \le c\Big\{
\E|\Gamma(s)|^2+\int_{s-\tau}^t\E|\Gamma(r)|^2 \d r\Big\}.
\end{split}
\end{equation}
This  leads to \eqref{eq10} by using   \eqref{v1} and noting that
\begin{equation*}
\E\|\Gamma_t\|_\8^2=\E\Big(\sup_{t-\tau\le s\le
t}|\Gamma(s)|^2\Big)\le \|\xi\|_\8^2+\E\Big(\sup_{(t-\tau)\vee0\le
s\le t}|\Gamma(s)|^2\Big).
\end{equation*}
\end{proof}

\begin{lem}\label{le1}
{\rm Under the assumptions of Theorem \ref{th0},
\begin{equation}\label{eq4}
 \E\|X_t(\xi,i)\|_\8^2\le c\,(1+\|\xi\|^2_\8),~~~~(\xi,i)\in\C\times\S.
\end{equation}
}
\end{lem}

\begin{proof}
By virtue of ({\bf H1}), for any $\gg>0$, there exists a  $c_\gg>0$
such that
\begin{equation}\label{eq60}
2\<\xi(0),b(\xi,i)\>+\|\si(\xi,i)\|_{\rm HS}^2\le c_\gg+
(\gg+\aa_i)|\xi(0)|^2+(\gg+\bb_i)\int_{-\tau}^0|\xi(\theta)|^2v(\d\theta)
\end{equation}
holds. Next, following the argument to derive \eqref{c2} and  making
use of \eqref{eq60}, we infer that
\begin{equation*}
\begin{split}
\e^{-\int_0^t(\gg+\aa_{\LL^{\omega_2}(s)}^\vv)\d
s}\E_{\P_1}|X^{\omega_2}(t)|^2
&\le
 c \,\|\xi\|_\8^2+\E_{\P_1}\Gamma_2^{\omega_2,\vv}(t)+c \int_0^t\e^{-\int_0^s(\gg+\aa_{\LL^{\omega_2}(r)}^\vv)\d
r}\d
s\\
&\quad+\int_0^t\Gamma_5^{\omega_2,\vv}(s)\e^{-\int_0^s(\dd+\aa_{\LL^{\omega_2}(r)}^\vv)\d
r}\E_{\P_1}|X^{\omega_2}(s)|^2\d s,
\end{split}
\end{equation*}
 where $\Gamma_2^{\omega_2,\vv}(t)$ is defined as in \eqref{x4} with
 writing $\int_{[-\tau,0]}|X^{\omega_2}(s+\theta)|^2v(\d\theta)$
 in lieu of $X^{\omega_2}_s$,
and
\begin{equation*}
\begin{split}
\Gamma^{\omega_2,\vv}_5(t):=\int_{-\tau}^0(\gg+\bb_{\LL^{\omega_2}(t-\theta)}^\vv)\e^{-\int_t^{t-\theta}(\gg+\aa_{\LL^{\omega_2}(r)}^\vv)\d
r}v(\d\theta).
\end{split}
\end{equation*}
Subsequently, an application of Gronwall's inequality yields that
\begin{equation*}
\begin{split}
\E_{\P_1}|X^{\omega_2}(t)|^2&\le c
\,\|\xi\|_\8^2+\E_{\P_1}\Gamma_2^{\omega_2,\vv}(t)+c
\int_0^t\e^{-\int_0^s(\gg+\aa_{\LL^{\omega_2}(r)}^\vv)\d r}\d s\\
&\quad+\int_0^t\Big(c
\,\|\xi\|_\8^2+\E_{\P_1}\Gamma_2^{\omega_2,\vv}(s)+c
\int_0^s\e^{-\int_0^u(\gg+\aa_{\LL^{\omega_2}(r)}^\vv)\d r}\d
u\Big)\\
&\quad\times\Gamma^{\omega_2,\vv}_5(s)\exp\Big(\int_s^t\Gamma^{\omega_2,\vv}_5(r)\d
r\Big)\d s.
\end{split}
\end{equation*}
Thus, following the lines to derive \eqref{b1}, we arrive at
\begin{equation}\label{v6}
\begin{split}
\E|X(t)|^2&\le
c\,\|\xi\|_\8^2\E\,\e^{\int_0^t(\gg+\aa_{\LL(s)}+\Gamma_4(s))\d
s}+c\int_0^t\E \e^{\int_s^t(\gg+\aa_{\LL(u)}+\Gamma_4(u))\d u}\d s,
\end{split}
\end{equation}
where
\begin{equation*}
\begin{split}
\Gamma_4(t):=\int_{-\tau}^0(\gg+\bb_{\LL(t-\theta)})\e^{-\int_t^{t-\theta}(\gg+\aa_{\LL(r)})\d
r}v(\d\theta),~~~~t>0.
\end{split}
\end{equation*}
Plugging the fact that
\begin{align*}
\int_s^t\Gamma_4(r)\d
r
&\le  c+\int_{-\tau}^0\e^{\hat{\alpha}\theta}v(\d\theta)
\int_s^t(\gg+\bb_{\LL^{\omega_2}(r)})\d r
\end{align*}
into \eqref{v6}  means that
\begin{equation*}
\begin{split}
\E|X(t)|^2 &\le
c\,\|\xi\|_\8^2\E\exp\Big(\int_0^t\Big(C_\gg+\aa_{\LL(s)}+\int_{-\tau}^0\e^{\hat{\alpha}\theta}v(\d\theta)\bb_{\LL(s)}\Big)\d
s\Big)\\
&\quad+c\int_0^t\E\exp\Big(\int_s^t\Big(C_\gg+\aa_{\LL(r)}
+\int_{-\tau}^0\e^{\hat{\alpha}\theta}v(\d\theta)\bb_{\LL(r)}\Big)\d
r\Big)\d s,
\end{split}
\end{equation*}
where $
C_\gg:=\gg\Big(1+\int_{-\tau}^0\e^{\hat{\alpha}\theta}v(\d\theta)\Big).
 $ Thus, with the aid of    \cite[Proposition 4.1]{B10} and by choosing
$\gg>0$ such that $C_\gg=\eta_2/2,$ we obtain from  $\eta_2>0$ that
\begin{equation}\label{eq8}
\E|X(t)|^2 \le  c\,\|\xi\|_\8^2\e^{-\eta_2 t/2}+  c\,
\int_0^t\e^{-\ff{\eta_2}{2}s}\d s\le c\,(1+\|\xi\|_\8^2).
\end{equation}
  Carrying out an analogous manner to derive
\eqref{v7}, we have
\begin{equation}\label{v8}
\E\Big(\sup_{s\le r\le t}|X(r)|^2\Big) \le c\Big\{1+\|\xi\|_\8^2+
\E|X(s)|^2+\int_{s-\tau}^t\E|X(r)|^2 \d r\Big\}.
\end{equation}
Thereby, \eqref{eq4} is now available from \eqref{eq8} and
\eqref{v8}.
\end{proof}

\noindent{{\bf Proof of Example \ref{exa1}}.}  \eqref{eq11} can be
regarded as the interactions between the following path-dependent
diffusion processes
\begin{equation}\label{eq13}
\d  X^{(i)}(t)=\{a_i  X^{(i)}(t)+b_i X^{(i)}(t-1)\}\d t+\si_i\d
W(t),~~~t>0,\ \  X^{(i)}_0=\xi\in\C,\  \ i=1,2.
\end{equation}
The characteristic equation associated with the deterministic
counterpart (i.e., $\si_i=0$) of \eqref{eq13} is
\begin{equation*}
\ll_i -\int_{-1}^0\e^{\ll_i s}\mu_i(\d
s)=:\triangle_{\mu_i}(\ll_i)=0,\ \ \ \ \ \ \ \ i=1,2,
\end{equation*}
where  $\mu_i(\cdot):=a_i\dd_0(\cdot)+b_i\dd_{-1}(\cdot)$, where
$\dd_x(\cdot)$ signifies Dirac's delta measure or unit mass at the
point $x$. By the variation-of-constants formula (see, e.g.,
\cite[Theorem 1]{AWM}), \eqref{eq13} can be expressed respectively
as
\begin{equation*}
X^{(i)}(t)=\GG_i(t)\xi(0)+b_i\int_{-1}^0 \GG_i(t-1-s)\xi(s)\d
s+\int_0^t \GG_i(t-s)\si_i\d W(s),~~~t>0,\ \  i=1,2.
\end{equation*}
Herein, $\GG_i(t)$ is the solution to the delay equation
\begin{equation}
\d Z^{(i)}(t)=\{a_i  Z^{(i)}(t)+b_i Z^{(i)}(t-1)\}\d t,~~~t>0
\end{equation}
with the initial value $Z^{(i)}(0)=1$ and $ Z^{(i)}(\theta)=0,
\theta\in[-1,0)$. In general,  $\GG_i(t)$ is called the fundamental
solution of \eqref{eq13} with $\si_i=0$.  It is readily to see that
 $\triangle_{\mu_1}(\ll)=0$ has a unique positive root.
Thus $\GG_0(t)\to\8$ as $t\uparrow\8$ so that $\E|
X^{(0)}(t)|\rightarrow\8$; see, e.g., \cite{RRV}. Hence,
$(X^{(0)}_t)$ does not admit an invariant probability measure. The
invariant probability measure of $(\LL(t))_{t\ge0}$ is
\begin{equation*}
\pi=(\pi_0,\pi_1)=\Big(\ff{\gg}{1+\gg},\ff{1}{1+\gg}\Big).
\end{equation*}
Observe that
\begin{equation*}
\begin{split}
|Q_2-\ll E|&= \left|\begin{array}{ccc}
  \aa-1-\ll & 1\\
  \gamma & \bb-\gamma-\ll\\
  \end{array}
  \right|\\
  &=(\aa-1-\ll)(p\bb-\gamma-\ll)-\gg\\
  &=\ll^2-(\aa+\bb-1-\gg)\ll+\aa\bb-(\aa\gg+\bb).
\end{split}
  \end{equation*}
As we know,  the characteristic equation $|Q_2-\ll E|=0$ has two
negative roots, $\ll_1$ and $\ll_2$, if and only if
\begin{equation*}
\begin{cases}
\ll_1+\ll_2=\aa+\bb-1-\gg<0\\
\ll_1\ll_2=\aa\bb-(\aa\gg+\bb)>0.
\end{cases}
\end{equation*}
Nevertheless, the inequalities above hold under \eqref{hui6}.

\section{Proof of Theorem \ref{th3}}\label{sec4}

Before proving Theorem \ref{th3}, we present an estimate on the
exponential functional of the discrete-time observations of the
Markov chain. This lemma plays a crucial role in the analyzing the
long-time behavior of the time discretization for
$(X_t(\xi,i_0),\LL(t))$ and is of interest by itself.

\begin{lem}\label{exp}
{\rm Let
 $K:\S\rightarrow\R$, and
$\displaystyle
Q_K=Q+ \mathrm{diag}(K_1, \cdots,K_N)$. Set
 \begin{equation*}
 \eta_K=
-\max_{\gamma\in\mathrm{spec}{(Q_K)}}\mathrm{Re}(\gamma).
\end{equation*}  Then there exist $\delta_0\in (0,1)$ and $c>0$ such that, for any $\delta\in (0,\delta_0)$,
\begin{equation}
\E\,\e^{\int_0^tK_{\LL(s_\dd)}\d s}\le c\,\e^{-\eta_K t/2},\quad \ \forall\,t>0.
\end{equation}

}
\end{lem}

\begin{proof}
By H\"older's inequality, it follows that
\begin{equation}\label{ap1}
\begin{split}
\E\,\e^{\int_0^tK_{\LL(s_\dd)}\d s}&=\E\,\e^{\int_0^tK_{\LL(s)}\d
s+\int_0^t(K_{\LL(s_\dd)}-K_{\LL(s)})\d s}\\
&\le\Big(\E\,\e^{(1+\vv)\int_0^tK_{\LL(s)}\d
s}\Big)^{\ff{1}{1+\vv}}\Big(\E\,\e^{\ff{1+\vv}{\vv}\int_0^t(K_{\LL(s_\dd)}-K_{\LL(s)})\d
s}\Big)^{\ff{\vv}{1+\vv}},~~~~~~\vv>0.
\end{split}
\end{equation}
Observe from \eqref{y1} that there exists $\dd_1\in(0,1)$ such that
for any $\3\in(0,\dd_1)$,
\begin{equation}\label{ap3}
\P(\LL(t+\triangle)=i|\LL(t)=i)=1+q_{ii}\3+o(\3),
\end{equation}
and that
\begin{equation}\label{ap4}
\P(\LL(t+\triangle)\neq i|\LL(t)=i)=\sum_{j\neq
i}(q_{ij}\3+o(\3))\le \max_{i\in\S}(-q_{ii})\3+o(\3).
\end{equation}
Utilizing Jensen's inequality and taking advantage of \eqref{ap3}
and \eqref{ap4}, we derive that for any $\dd\in(0,\dd_1),$
\begin{equation}\label{ap5}
\begin{split}
&\E\,\Big(\e^{\ff{1+\vv}{\vv}\int_{i\dd}^{(i+1)\dd\wedge
t}(K_{\LL(i\dd)}-K_{\LL(s)})\d s}\Big|\LL(i\dd)\Big)\\
&\le\ff{1}{(i+1)\dd\wedge t-i\dd}\int_{i\dd}^{(i+1)\dd\wedge
t}\E\,\Big(\e^{\ff{1+\vv}{\vv}((i+1)\dd\wedge
t-i\dd)(K_{\LL(i\dd)}-K_{\LL(s)})}\Big|\LL(i\dd)\Big)\d
s\\
&=\ff{\sum_{j\in \S}{\bf1}_{\{\LL(i\dd)=j\}}}{(i+1)\dd\wedge
t-i\dd}\int_{i\dd}^{(i+1)\dd\wedge
t}\E\,\Big(\e^{\ff{1+\vv}{\vv}((i+1)\dd\wedge
t-i\dd)(K_j-K_{\LL(s)})}\Big|\LL(i\dd)=j\Big)\d s\\
&=\ff{\sum_{j\in \S}{\bf1}_{\{\LL(i\dd)=j\}}}{(i+1)\dd\wedge
t-i\dd}\int_{i\dd}^{(i+1)\dd\wedge
t}\E\,({\bf1}_{\{\LL(s)=j\}}|\LL(i\dd)=j)\d
s\\
&\quad+\ff{\sum_{j\in \S}{\bf1}_{\{\LL(i\dd)=j\}}}{(i+1)\dd\wedge
t-i\dd}\int_{i\dd}^{(i+1)\dd\wedge
t}\E\,\Big(\e^{\ff{1+\vv}{\vv}((i+1)\dd\wedge
t-i\dd)(K_j-K_{\LL(s)})}{\bf1}_{\{\LL(s)\neq
j\}}\Big|\LL(i\dd)=j\Big)\d s\\
&\le\ff{\sum_{j\in \S}{\bf1}_{\{\LL(i\dd)=j\}}}{(i+1)\dd\wedge
t-i\dd}\int_{i\dd}^{(i+1)\dd\wedge
t}\E\,({\bf1}_{\{\LL(s)=j\}}|\LL(i\dd)=j)\d
s\\
&\quad+\e^{\ff{2(1+\vv)\check{K}\dd}{\vv}}\ff{\sum_{j\in
\S}{\bf1}_{\{\LL(i\dd)=j\}}}{(i+1)\dd\wedge
t-i\dd}\int_{i\dd}^{(i+1)\dd\wedge
t}\P\,({\bf1}_{\{\LL(s)\neq j\}}|\LL(i\dd)=j)\d s\\
&\le\ff{\sum_{j\in \S}{\bf1}_{\{\LL(i\dd)=j\}}}{(i+1)\dd\wedge
t-i\dd}\int_{i\dd}^{(i+1)\dd\wedge
t}(1+q_{jj}(s-i\dd)+o(s-i\dd))\d s\\
&\quad+\e^{\ff{2(1+\vv)\check{K}\dd}{\vv}}\ff{\sum_{j\in
\S}{\bf1}_{\{\LL(i\dd)=j\}}}{(i+1)\dd\wedge
t-i\dd}\int_{i\dd}^{(i+1)\dd\wedge
t}\Big(\max_{i\in\S}(-q_{ii})(s-i\dd)+o(s-i\dd)\Big)\d s\\
&\le1+\ff{\max_{i\in\S}(-q_{ii})}{2}\,\dd\,\e^{\ff{2(1+\vv)\check{K}\dd}{\vv}}+o(\dd),
\end{split}
\end{equation}
where $\check{K}:=\max_{i\in\S}|K_i|$. By the property of
conditional expectation, we deduce from \eqref{ap5} that
\begin{equation}\label{ap2}
\begin{split}
&\E\,\e^{\ff{1+\vv}{\vv}\int_0^t(K_{\LL(s_\dd)}-K_{\LL(s)})\d
s}\\
&=\E\,\e^{\ff{1+\vv}{\vv}\sum_{i=0}^{\lfloor
t/\dd\rfloor}\int_{i\dd}^{(i+1)\dd\wedge
t}(K_{\LL(i\dd)}-K_{\LL(s)})\d
s}\\&=\E\Big(\E\,\Big(\e^{\ff{1+\vv}{\vv}\sum_{i=0}^{\lfloor
t/\dd\rfloor}\int_{i\dd}^{(i+1)\dd\wedge
t}(K_{\LL(i\dd)}-K_{\LL(s)})\d s}\Big|\LL(t_\dd)\Big)\Big)\\
&=\E\Big(\e^{\ff{1+\vv}{\vv}\sum_{i=0}^{\lfloor
t/\dd\rfloor-1}\int_{i\dd}^{(i+1)\dd}(K_{\LL(i\dd)}-K_{\LL(s)})\d
s}\E\,\Big(\e^{\ff{1+\vv}{\vv}\int_{t_\dd}^{(t_\dd+\dd)\wedge
t}(K_{\LL(t_\dd)}-K_{\LL(s)})\d s}\Big|\LL(t_\dd)\Big)\Big)\\
&\le\Big(1+\ff{\max_{i\in\S}(-q_{ii})}{2}\,\dd\,\e^{\ff{2(1+\vv)\check{K}\dd}{\vv}}+o(\dd)\Big)\E\Big(\e^{\ff{1+\vv}{\vv}\sum_{i=0}^{\lfloor
t/\dd\rfloor-1}\int_{i\dd}^{(i+1)\dd}(K_{\LL(i\dd)}-K_{\LL(s)})\d
s}\Big)\\
&\le\Big(1+\ff{\max_{i\in\S}(-q_{ii})}{2}\,\dd\,\e^{\ff{2(1+\vv)\check{K}\dd}{\vv}}+o(\dd)\Big)^{\lfloor
t/\dd\rfloor+1},~~~~\dd\in(0,\dd_1),
\end{split}
\end{equation}
where $t_\dd:=\lfloor t/\dd\rfloor\dd.$ For any $c_1,c_2>0$,
\begin{equation*}
\lim_{\dd\rightarrow0}\ff{1}{\dd}\ln(1+c_1\dd\,\e^{c_2\dd}+o(\delta))=c_1.
\end{equation*}
So, there exists $\dd_2=\dd_2(c_1,c_2)\in(0,1)$ such that
\begin{equation}\label{ap9}
\ff{1}{\dd}\ln(1+c_1\dd\,\e^{c_2\dd})\le
2\,c_1,~~~~~\dd\in(0,\dd_2).
\end{equation}
Note that $\delta_2$ depends on $c_2$, and $\delta_2$ decreases as $c_2$ increasing.

According to \eqref{ap9}, there exists $\delta_2=\delta_2(\vv)$ so that for any $\dd\in(0,\dd_1\wedge\dd_2)$,
\begin{equation}\label{ap10}
\begin{split}
&\Big(1+\ff{\max_{i\in\S}(-q_{ii})}{2}\,\dd\,\e^{\ff{2(1+\vv)\check{K}\dd}{\vv}}+o(\dd)\Big)^{\ff{\vv(\lfloor
t/\dd\rfloor+1)}{1+\vv}}\\
&=\exp\Big(\ff{\vv(\lfloor
t/\dd\rfloor+1)}{1+\vv}\ln\Big(1+\ff{\max_{i\in\S}(-q_{ii})}{2}\,\dd\,\e^{\ff{2(1+\vv)\check{K}\dd}{\vv}}+o(\dd)\Big)\Big)\\
&\le\exp\Big(\ff{\vv (t+\dd)}{1+\vv}\ff{1}{\dd}\ln\Big(1+\ff{\max_{i\in\S}(-q_{ii})}{2}\,\dd\,\e^{\ff{2(1+\vv)\check{K}\dd}{\vv}}+o(\dd)\Big)\Big)\\
&\le\exp\Big(\ff{\vv}{1+\vv}
\max_{i\in\S}(-q_{ii})\Big)\exp\Big(\ff{\vv}{1+\vv}
\max_{i\in\S}(-q_{ii})t\Big).
\end{split}
\end{equation}
Taking \eqref{ap2} and \eqref{ap10} into consideration, we deduce
from \eqref{ap1} that
\begin{equation}\label{ap11}
\E\,\e^{\int_0^tK_{\LL(s_\dd)}\d s} \le\exp\Big(\ff{\vv}{1+\vv}
\max_{i\in\S}(-q_{ii})\Big)\Big(\E\,\e^{(1+\vv)\int_0^tK_{\LL(s)}\d
s}\Big)^{\ff{1}{1+\vv}}\exp\Big(\ff{\vv}{1+\vv}
\max_{i\in\S}(-q_{ii})t\Big).
\end{equation}
By virtue of \cite[Theorem 1.5 \& Proposition 4.1]{B10},  there exist    $\vv_0\in(0,1)$ sufficiently small
and $c>0$ such that
\begin{equation*}
\E\,\e^{(1+\vv)\int_0^tK_{\LL(s)}\d s}\le c\,\e^{-2\eta_K
t/3},~~~~~\vv\in(0,\vv_0),\  t>0.
\end{equation*}
Inserting this into \eqref{ap11}   yields that
\begin{equation*}
\E\,\e^{\int_0^tK_{\LL(s_\dd)}\d s} \le c^{\frac{1}{1+\vv}}\exp\Big(\ff{\vv}{1+\vv}
\max_{i\in\S}(-q_{ii})\Big)\exp\Big(-\ff{2\eta_K/3-
\max_{i\in\S}(-q_{ii})\vv}{1+\vv}t\Big),\quad t>0.
\end{equation*}
Thus, the desired assertion follows by  taking first $\vv$ small enough and then $\delta\in(0,\delta_1\wedge \delta_2(\vv))$.
\end{proof}

\begin{rem}
  The crucial point of   lemma \ref{exp} is that the choice of $\delta_0$ is independent of time $t$. Otherwise, it is easy to obtain a
  similar estimate for a time-dependent $\delta_0$ by using the dominated convergence theorem. Indeed,
  \[\lim_{\delta\rightarrow 0} \E \e^{\int_0^t K_{\LL(s(\dd))}\d s} =\E\e^{\int_0^t K_{\LL(s-)}\d s }=\E \e^{\int_0^t K_{\LL(s)}\d s }.\]
  Then, applying \cite[Theorem 1.5 \& Proposition 4.1]{B10} will yield the estimate for a given time $t$.
\end{rem}

Next, we provide two crucial lemmas on distance between the
$\C$-valued stochastic processes $Y_{t_\dd}$ starting from different
points and its uniform boundedness under some appropriate
assumptions.

\begin{lem}\label{lem4.1}
{\rm Let the assumptions of Theorem \ref{th4} be satisfied and
suppose further \eqref{w3} holds. Then, there exist $\dd_0\in(0,1)$
and $\aa>0$ such that
\begin{equation}\label{w6}
\E\|Y_{t_\dd}(\xi,i)-Y_{t_\dd}(\eta,i)\|_\8^2\le c\,\e^{-\aa
t}\|\xi-\eta\|^2_\8, ~~~~t\ge\tau+1, ~~\dd\in(0,\dd_0)
\end{equation}
for any $\xi,\eta\in\C$ and $i\in\S.$

}
\end{lem}

\begin{proof}
Hereinafter, we assume that $t\ge\tau+1.$ Fix $\omega_2\in\OO_2$ and
let $(Y^{\omega_2}(t))$ solve the following SDE
\begin{equation*}
\d Y^{\omega_2}(t) =b(Y_{t_\dd}^{\omega_2},\LL^{\omega_2}(t_\dd))\d
t+\si(\LL^{\omega_2}(t_\dd))\d \omega_1(t)
\end{equation*}
with the initial value $Y^{\omega_2}(s)=\xi(s), s\in[-\tau,0],$ and
$\LL^{\omega_2}(0)=i\in\S.$ For notational simplicity, set
\begin{equation}\label{a0}
\Upsilon^{\omega_2}(t):=Y^{\omega_2}(t;\xi,i)-Y^{\omega_2}(t;\eta,i).
\end{equation}
First of all, we claim that
\begin{equation}\label{z1}
\begin{split}
\e^{-\int_0^t\aa_{\LL^{\omega_2}(s_\dd)}\d
s}|\Upsilon^{\omega_2}(t)|^2&=|\Upsilon^{\omega_2}(0)|^2+\int_0^t\e^{-\int_0^s\aa_{\LL^{\omega_2}(r_\dd)}\d
r}\Big\{-\aa_{\LL^{\omega_2}(s_\dd)}|\Upsilon^{\omega_2}(s)|^2\\
&\quad+2\<\Upsilon^{\omega_2}(s),b(Y_{s_\dd}^{\omega_2}(\xi,i),\LL^{\omega_2}(s_\dd))-b(Y_{s_\dd}^{\omega_2}(\eta,i),\LL^{\omega_2}(s_\dd))\>\Big\}\d
s.
\end{split}
\end{equation}
For any $t\in(0,\dd)$, by It\^o's formula, we have
\begin{equation*}
\begin{split}
\e^{-\int_0^t\aa_{\LL^{\omega_2}(s_\dd)}\d
s}|\Upsilon^{\omega_2}(t)|^2&=\e^{-\aa_{\LL^{\omega_2}(0)}t}|\Upsilon^{\omega_2}(t)|^2\\
&=|\Upsilon^{\omega_2}(0)|^2+\int_0^t\e^{-\aa_{\LL^{\omega_2}(0)}s}\Big\{-\aa_{\LL^{\omega_2}(0)}|\Upsilon^{\omega_2}(s)|^2\\
&\quad+2\<\Upsilon^{\omega_2}(s),b(Y_{0}^{\omega_2}(\xi,i),\LL^{\omega_2}(0))-b(Y_{0}^{\omega_2}(\eta,i),\LL^{\omega_2}(0))\>\Big\}\d
s\\
&=|\Upsilon^{\omega_2}(0)|^2+\int_0^t\e^{-\int_0^s\aa_{\LL^{\omega_2}(r_\dd)}\d
r}\Big\{-\aa_{\LL^{\omega_2}(s_\dd)}|\Upsilon^{\omega_2}(s)|^2\\
&\quad+2\<\Upsilon^{\omega_2}(s),b(Y_{s_\dd}^{\omega_2}(\xi,i),\LL^{\omega_2}(s_\dd))-b(Y_{s_\dd}^{\omega_2}(\eta,i),\LL^{\omega_2}(s_\dd))\>\Big\}\d
s.
\end{split}
\end{equation*}
Accordingly, \eqref{z1} holds for any $t\in[0,\dd].$ Next, we assume
that \eqref{z1} is true for any $t\in[(k-1)\dd,k\dd).$ For any
$t\in[k\dd,(k+1)\dd)$, It\^o's formula yields that
\begin{equation*}
\begin{split}
\e^{-\aa_{\LL^{\omega_2}(k\dd)}(t-k\dd)}|\Upsilon^{\omega_2}(t)|^2&=
|\Upsilon^{\omega_2}(k\dd)|^2+\int_{k\dd}^t\e^{-\aa_{\LL^{\omega_2}(k\dd)}(s-k\dd)}\Big\{-\aa_{\LL^{\omega_2}(k\dd)}|\Upsilon^{\omega_2}(s)|^2\d
s\\
&\quad+2\<\Upsilon^{\omega_2}(s),b(Y_{k\dd}^{\omega_2}(\xi,i),\LL^{\omega_2}(k\dd))-b(Y_{k\dd}^{\omega_2}(\eta,i),\LL^{\omega_2}(k\dd))\>\Big\}\d
s.
\end{split}
\end{equation*}
Multiplying both sides by
$\e^{-\int_0^{k\dd}(\gg+\aa_{\LL^{\omega_2}(s_\dd)})\d s}$ and
applying \eqref{z1} with $t=k\dd$ leads to
\begin{equation*}
\begin{split}
\e^{-\int_0^t\aa_{\LL^{\omega_2}(s_\dd)}\d
s}|\Upsilon^{\omega_2}(t)|^2&=\e^{-\int_0^{k\dd}(\gg+\aa_{\LL^{\omega_2}(s_\dd)})\d
s}
|\Upsilon^{\omega_2}(k\dd)|^2+\int_{k\dd}^t\e^{-\int_0^s\aa_{\LL^{\omega_2}(r_\dd)}\d
r}\Big\{-\aa_{\LL^{\omega_2}(s_\dd)}|\Upsilon^{\omega_2}(s)|^2\d
s\\
&\quad+2\<\Upsilon^{\omega_2}(s),b(Y_{s_\dd}^{\omega_2}(\xi,i),\LL^{\omega_2}(s_\dd))-b(Y_{s_\dd}^{\omega_2}(\eta,i),\LL^{\omega_2}(s_\dd))\>\Big\}\d
s\\
&=|\Upsilon^{\omega_2}(0)|^2+\int_0^t\e^{-\int_0^s\aa_{\LL^{\omega_2}(r_\dd)}\d
r}\Big\{-\aa_{\LL^{\omega_2}(s_\dd)}|\Upsilon^{\omega_2}(s)|^2\\
&\quad+2\<\Upsilon^{\omega_2}(s),b(Y_{s_\dd}^{\omega_2}(\xi,i),\LL^{\omega_2}(s_\dd))-b(Y_{s_\dd}^{\omega_2}(\eta,i),\LL^{\omega_2}(s_\dd))\>\Big\}\d
s.
\end{split}
\end{equation*}
Thereby, \eqref{z1} follows immediately. It is readily to see from
\eqref{w3} that
\begin{equation}\label{w5}
\begin{split}
|\Upsilon^{\omega_2}(t)-\Upsilon^{\omega_2}(t_\dd)|&=|b(Y_{t_\dd}^{\omega_2}(\xi,i),\LL^{\omega_2}(t_\dd))-b(Y_{t_\dd}^{\omega_2}(\eta,i),\LL^{\omega_2}(t_\dd))|\dd\\
&\le L_0\|\Upsilon^{\omega_2}_{t_\dd}\|_\8\dd.
\end{split}
\end{equation}
By virtue of \eqref{z1} and   ({\bf A}), it follows that
\begin{equation}\label{w0}
\begin{split}
&\e^{-\int_0^t\aa_{\LL^{\omega_2}(s_\dd)}\d
s}|\Upsilon^{\omega_2}(t)|^2\\&=|\Upsilon^{\omega_2}(0)|^2+\int_0^t\e^{-\int_0^s\aa_{\LL^{\omega_2}(r_\dd)}\d
r}\Big\{-\aa_{\LL^{\omega_2}(s_\dd)}|\Upsilon^{\omega_2}(s)|^2\\
&\quad+2\<\Upsilon^{\omega_2}(s),b(Y_{s_\dd}^{\omega_2}(\xi,i),\LL^{\omega_2}(s_\dd))-b(Y_{s_\dd}^{\omega_2}(\eta,i),\LL^{\omega_2}(s_\dd))\>\Big\}\d
s\\
&\le|\Upsilon^{\omega_2}(0)|^2+\int_0^t\e^{-\int_0^s\aa_{\LL^{\omega_2}(r_\dd)}\d
r}\Big\{\aa_{\LL^{\omega_2}(s_\dd)}(|\Upsilon^{\omega_2}(s_\dd)|^2-|\Upsilon^{\omega_2}(s)|^2)+\bb_{\LL^{\omega_2}(s_\dd)}\|\Upsilon^{\omega_2}_{s_\dd}\|^2_\8\\
&\quad+2\<\Upsilon^{\omega_2}(s)-\Upsilon^{\omega_2}(s_\dd),b(Y_{s_\dd}^{\omega_2}(\xi,i),\LL^{\omega_2}(s_\dd))-b(Y_{s_\dd}^{\omega_2}(\eta,i),\LL^{\omega_2}(s_\dd))\>\Big\}\d
s\\
&\le|\Upsilon^{\omega_2}(0)|^2+\int_0^t\e^{-\int_0^s\aa_{\LL^{\omega_2}(r_\dd)}\d
r}\Big\{\ff{1+2\check{\aa}}{\sqrt{\dd}}|\Upsilon^{\omega_2}(s)-\Upsilon^{\omega_2}(s_\dd)|^2\\
&\quad+(\check{\aa}\sqrt{\dd}+L_0^2\sqrt{\dd}+\bb_{\LL^{\omega_2}(s_\dd)})
\|\Upsilon_{s_\dd}^{\omega_2}\|^2_\8\Big\}\d s\\
&\le|\Upsilon^{\omega_2}(0)|^2+\int_0^t(\rr\sqrt{\dd}+\bb_{\LL^{\omega_2}(s_\dd)})
\e^{-\int_0^s\aa_{\LL^{\omega_2}(r_\dd)}\d
r}\|\Upsilon_{s_\dd}^{\omega_2}\|^2_\8\d s,
\end{split}
\end{equation}
where $\rho:=2(1+\check{a})L_0^2+\check{\aa}$, and in the penultimate
display we have used \eqref{w5}. Observe that
\begin{equation*}
\begin{split}
\Pi^{\omega_2}(t):&=\e^{-\int_0^t\aa_{\LL^{\omega_2}(s_\dd)}\d s}
\Big(\sup_{t-\tau-\dd\le s\le t}|\Upsilon^{\omega_2}(s)|^2
\Big)\\
&\le\e^{-\hat{\alpha}(\tau+\dd)} \Big(\sup_{t-\tau-\dd\le s\le
t}\Big(\e^{-\int_0^s\aa_{\LL^{\omega_2}(r_\dd)}\d
r}|\Upsilon^{\omega_2}(s)|^2\Big)\Big)
\end{split}
\end{equation*}
and that
\begin{equation*}
\|\Upsilon_{t_\dd}^{\omega_2}\|_\8=\sup_{-\tau\le\theta\le0}|\Upsilon_{t_\dd}^{\omega_2}(\theta;\xi,i)|=\max_{-N\le
i\le0}|\Upsilon^{\omega_2}(t_\dd+i\dd;\xi,i)|\le\sup_{t-\tau-\dd\le
s\le t}|\Upsilon^{\omega_2}(s)|,
\end{equation*}
by \eqref{w2}  and $Y(t)=\xi(-\tau)$ for any $ t\in[-\tau-1,-\tau)
$. We therefore obtain from \eqref{w0}  that
\begin{equation*}
\Pi^{\omega_2}(t)\le
c\,\|\Upsilon^{\omega_2}_0\|^2_\8+\e^{-\hat{\alpha}(\tau+\dd)}\int_0^t(\rho\sqrt{\dd}+\bb_{\LL^{\omega_2}(s_\dd)})
\Pi^{\omega_2}(s)\d s.
\end{equation*}
This, together with Gronwall's inequality, implies that
\begin{equation*}
\E\|Y_{t_\dd}(\xi,i_0)-Y_{t_\dd}(\eta,i_0)\|_\8^2\le
c\,\|\xi-\eta\|^2_\8\e^{\e^{-\hat{\alpha}(\tau+\dd)}\rho\sqrt{\dd} t}
\E\,\e^{\int_0^t(\aa_{\LL(s_\dd)}+\e^{-\hat{\alpha}(\tau+\dd)}\bb_{\LL(s_\dd)})\d
s}.
\end{equation*}
Thus, according to Lemma \ref{exp}, it holds
\begin{equation*}
\E\|Y_{t_\dd}(\xi,i_0)-Y_{t_\dd}(\eta,i_0)\|_\8^2\le
c\,\|\xi-\eta\|^2_\8\e^{\e^{-\hat{\alpha}(\tau+\dd)}\rho\sqrt{\dd} t}
\e^{-2\eta_1t/3}
\end{equation*}
for $\dd\in(0,\dd_1)$ with some $\dd_1\in(0,1)$ sufficiently small.
As $\eta_1>0$, there exists $\dd_2\in(0,\dd_1)$ such that
$\e^{-\hat{\alpha}(\tau+\dd)}\rho\sqrt{\dd}<\eta_1$ for
$\dd\in(0,\dd_2)$. As a consequence, \eqref{w6} holds for
$\dd\in(0,\dd_2)$.
\end{proof}

\begin{lem}\label{lem4.2}
{\rm Under the assumptions of Lemma \ref{lem4.1}, there exists some
$\dd_0\in(0,1) $ such that
\begin{equation}\label{b2}
\E\|Y_{t_\dd}(\xi,i)\|_\8^2\le c\, (1+\|\xi\|^2_\8), ~~~~t\ge\tau+1,
~~\dd\in(0,\dd_0)
\end{equation}
for any $(\xi,i)\in\C\times\S.$

}
\end{lem}

\begin{proof}
Below, we assume $t\ge\tau+1$. Carrying out the procedure to gain
\eqref{z1}, we have
\begin{equation}\label{b11}
\begin{split}
&\e^{-\int_0^t(\gg+\aa_{\LL^{\omega_2}(s_\dd)})\d
s}|Y^{\omega_2}(t)|^2\\&=|\xi(0)|^2+\int_0^t\e^{-\int_0^s(\gg+\aa_{\LL^{\omega_2}(r_\dd)})\d
r}\Big\{-(\gg+\aa_{\LL^{\omega_2}(s_\dd)})|Y^{\omega_2}(s)|^2\\
&\quad+2\<Y^{\omega_2}(s),b(Y_{s_\dd}^{\omega_2},\LL^{\omega_2}(s_\dd))\>+\|\si(\LL^{\omega_2}(s_\dd))\|_{\rm
HS}^2\Big\}\d s\\
&\quad+2\int_0^t\e^{-\int_0^s(\gg+\aa_{\LL^{\omega_2}(r_\dd)})\d
r}\<Y^{\omega_2}(s),\si(\LL^{\omega_2}(s_\dd))\d \omega_1(s)\>,
\end{split}
\end{equation}
where $\gg>0$ is introduced in \eqref{eq6} and is to be determined.
Thanks to \eqref{w3}, it holds that
\begin{equation}\label{w1}
\begin{split}
|Y^{\omega_2}(t)-Y^{\omega_2}({t_\dd})|^2 &\le
2L^2_0\dd\|Y_{t_\dd}^{\omega_2}\|_\8^2+c\,|\omega_1(t)-\omega_1(t_\dd)|^2.
\end{split}
\end{equation}
and, from   \eqref{w2} and $Y(t)=\xi(-\tau)$ for any $
t\in[-\tau-1,-\tau) $, that
\begin{equation}\label{w7}
\|Y_{t_\dd}^{\omega_2}\|_\8\le\sup_{t-\tau-\dd\le s\le
t}|Y^{\omega_2}(s)|.
\end{equation}
%
%
%
%
%
Thus, by combining \eqref{eq6} with \eqref{b11}-\eqref{w7}, it
follows that
\begin{equation}\label{b12}
\begin{split}
&\e^{-\int_0^t(\gg+\aa_{\LL^{\omega_2}(s_\dd)})\d
s}|Y^{\omega_2}(t)|^2\\
&\le|\xi(0)|^2+\int_0^t\e^{-\int_0^s(\gg+\aa_{\LL^{\omega_2}(r_\dd)})\d
r}\Big\{c+((\gg+\check{a})\sqrt{\dd}+\bb_{\LL^{\omega_2}(s_\dd)})\|Y_{s_\dd}^{\omega_2}\|_\8^2\\
&\quad+\ff{1+\gg+2\check{a}}{\sqrt{\dd}}|Y^{\omega_2}(s)-Y^{\omega_2}(s_\dd)|^2+\sqrt{\dd}
|b(Y_{s_\dd}^{\omega_2},\LL^{\omega_2}(s_\dd))|^2\Big\}\d s+\Theta^{\omega_2}(t)\\
&\le|\xi(0)|^2+\int_0^t\e^{-\int_0^s(\gg+\aa_{\LL^{\omega_2}(r_\dd)}^\vv)\d
r}\Big\{c+\ff{c}{\sqrt{\dd}}\,|\omega_1(t)-\omega_1(t_\dd)|^2\\
&\quad+(\vartheta\sqrt{\dd}+\bb_{\LL^{\omega_2}(s_\dd)})\sup_{s-\tau-\dd\le
r\le s}|Y(r)|^2\Big\}\d s+\Theta^{\omega_2}(t)
\end{split}
\end{equation}
where $\vartheta:=\gg+\check{a}+2L^2_0(2+\gg+2\check{a})$, and
$$
\Theta^{\omega_2}(t):=2\int_0^t\e^{-\int_0^s(\gg+\aa_{\LL^{\omega_2}(r_\dd)})\d
r}\<Y^{\omega_2}(s),\si(\LL^{\omega_2}(s_\dd))\d \omega_1(s)\>. $$
Following the argument to derive \eqref{b13},  for   $0\le s\le t$
with $t-s\in[0, \tau+\dd]$ and
  $\kk\in(0,1)$, which is also to be determined, we have
\begin{equation}\label{b14}
\begin{split}
\E_{\P_1}\Big(\sup_{s\le r\le t}\Theta^{\omega_2}(r)\Big)
&\le\kk\,\e^{\hat{\alpha}(\tau+\dd)}\Pi^{\omega_2}(t)+c\,\e^{-\int_0^t(\dd+\aa_{\LL^{\omega_2}(r_\dd)})\d
r},
\end{split}
\end{equation}
and observe that
\begin{equation*}
\begin{split}
\Pi^{\omega_2}(t):&=\e^{-\int_0^t(\gg+\aa_{\LL^{\omega_2}(s_\dd)})\d
s}\E_{\P_1}\Big(\sup_{t-\tau-\dd\le s\le t}|Y^{\omega_2}(s)|^2
\Big)\\
&\le \e^{-\hat{\alpha}(\tau+\dd)}\E_{\P_1}\Big(\sup_{t-\tau-\dd\le s\le
t}\Big(\e^{-\int_0^s(\gg+\aa_{\LL^{\omega_2}(r)})\d
r}|Y^{\omega_2}(s)|^2\Big)\Big).
\end{split}
\end{equation*}
Hence, we deduce from \eqref{b12} and \eqref{b14} that
\begin{equation*}
\begin{split}
\Pi^{\omega_2}(t) &\le
\ff{\e^{-\hat{\alpha}(\tau+\dd)}}{1-\kk}\Big\{c\,\|\xi\|^2_\8+c\,\e^{-\int_0^t(\gg+\aa_{\LL^{\omega_2}(r_\dd)})\d
r}+c\int_0^t\e^{-\int_0^s(\gg+\aa_{\LL^{\omega_2}(r_\dd)}^\vv)\d
r}\d s\\
&\quad+\int_0^t(\vartheta\sqrt{\dd}+\bb_{\LL^{\omega_2}(s_\dd)})\Pi^{\omega_2}(s)\d
s\Big\}.
\end{split}
\end{equation*}
Thus, an application of Gronwall's inequality enables us to get
\begin{equation}\label{b15}
\begin{split}
\Pi^{\omega_2}(t) &\le
\ff{\e^{-\hat{\alpha}(\tau+\dd)}}{1-\kk}\Big\{c\,\|\xi\|^2_\8+c\,\e^{-\int_0^t(\gg+\aa_{\LL^{\omega_2}(r_\dd)})\d
r}+c\int_0^t\e^{-\int_0^s(\gg+\aa_{\LL^{\omega_2}(r_\dd)}^\vv)\d
r}\d s\Big\}\\
&\quad+\ff{\e^{-\hat{\alpha}(\tau+\dd)}}{1-\kk}\int_0^t\Big\{c\,\|\xi\|^2_\8+c\,\e^{-\int_0^s(\gg+\aa_{\LL^{\omega_2}(r_\dd)})\d
r}+c\int_0^s\e^{-\int_0^u(\gg+\aa_{\LL^{\omega_2}(r_\dd)}^\vv)\d
r}\d u\Big\}\\
&\quad\times\Phi^{\omega_2}(s_\dd)
\exp\Big(\int_s^t\Phi^{\omega_2}(r_\dd)\d r\Big)\d s,
\end{split}
\end{equation}
in which
\begin{equation*}
\Phi^{\omega_2}(t_\dd):=\ff{\e^{-\hat{\alpha}(\tau+\dd)}(\vartheta\dd+\bb_{\LL^{\omega_2}(t_\dd)})}{1-\kk}.
\end{equation*}
For any $0\le s\le t$, set
\begin{equation*}
\Upsilon^{\omega_2}(s,t):=\int_s^t\Phi^{\omega_2}(u_\dd)
\exp\Big(\int_u^t\Phi^{\omega_2}(r_\dd)\d r\Big)\d u.
\end{equation*}
For any $0\le s\le t$, there exist   integers $j,k>0$ such that
$s\in[j\dd,(j+1)\dd)$ and $t\in[k\dd,(k+1)\dd)$. If $j=k,$ then we
obtain that
\begin{equation}\label{b16}
\Upsilon^{\omega_2}(s,t)=\int_s^t\Phi^{\omega_2}(k\dd)
\e^{\Phi^{\omega_2}(k\dd)(t-u)}\d
u=\e^{\Phi^{\omega_2}(k\dd)(t-s)}-1=\exp\Big(\int_s^t\Phi^{\omega_2}(r_\dd)\d
r\Big)-1.
\end{equation}
In the sequel, we assume that $j<k.$ Observe that
\begin{equation}\label{b17}
\begin{split}
\Upsilon^{\omega_2}(s,t)
&=\e^{\Phi^{\omega_2}(k\dd)(t-k\dd)}\Upsilon(s,k\dd)+
\e^{\Phi^{\omega_2}(k\dd)(t-k\dd)}-1\\
&=\e^{\Phi^{\omega_2}(k\dd)(t-k\dd)}\Big\{\e^{\Phi^{\omega_2}((k-1)\dd)\dd}\Upsilon(s,(k-1)\dd)+\e^{\Phi^{\omega_2}((k-1)\dd)\dd}-1\Big\}+ \e^{\Phi^{\omega_2}(k\dd)(t-k\dd)}-1\\
&=\e^{\Phi^{\omega_2}(k\dd)(t-k\dd)+\Phi^{\omega_2}((k-1)\dd)\dd}\Upsilon(s,(k-1)\dd)+ \e^{\Phi^{\omega_2}(k\dd)(t-k\dd)+\Phi^{\omega_2}((k-1)\dd)\dd}-1\\
&=\cdots\\
&=\e^{\Phi^{\omega_2}(k\dd)(t-k\dd)+\Phi^{\omega_2}((k-1)\dd)\dd+\cdots+\Phi^{\omega_2}(j\dd)((j+1)\dd-s)}-1\\
&=\exp\Big(\int_s^t\Phi^{\omega_2}(r_\dd)\d r\Big)-1.
\end{split}
\end{equation}
Subsequently, taking \eqref{b15}-\eqref{b17} and Fubini's theorem
into account, we deduce that
\begin{equation*}
\begin{split}
\E\Big(\sup_{t-\tau-\dd\le s\le t}|Y(s)|^2 \Big)\le
c\,\E\Big\{1+\e^{\int_0^t(\gg+\aa_{\LL(s_\dd)}+\Phi(s_\dd))\d
s}+\int_0^t\e^{\int_s^t(\gg+\aa_{\LL(r_\dd)}+\Phi(r_\dd))\d r}\d
s\Big\},
\end{split}
\end{equation*}
where
\begin{equation*}
\Phi(t_\dd):=\ff{\e^{-\hat{\alpha}(\tau+\dd)}(\vartheta\sqrt{\dd}+\bb_{\LL(t_\dd)})}{1-\kk}.
\end{equation*}
Thus,  with the help of   $\eta_1>0$, \eqref{b2} follows from Lemma
\ref{exp} and by taking $\gg,\dd,\kk\in(0,1)$ sufficiently small.
\end{proof}

Now we are ready to finish the proof of Theorem \ref{th3}.

\begin{proof}
With Lemmas \ref{lem4.1} and \ref{lem4.2} in hand, we can complete
the argument of Theorem \ref{th3} by mimicking the proof of Theorem
\ref{th0}.
\end{proof}

\section{Proof of Theorem \ref{th5}}\label{sec5}
Before we complete the proof of Theorem \ref{th5}, let's make some
preparations. For any $t\ge0$, let $\mathcal
{F}_t=\si((W(u),\LL(u)),0\le u\le t)\vee\mathcal {N}$, where
$\mathcal {N}$ stands for the set of all $\P$-null sets in $\F $.

\begin{lem}\label{Markov}
{\rm $(Y_{k\dd},\LL(k\dd))$ is a homogeneous Markov chain, i.e.,
\begin{equation}\label{m2}
\begin{split}
&\P((Y_{(k+1)\dd},\LL((k+1)\dd))\in
A\times\{j\}|(Y_{k\dd},\LL(k\dd))  =(\xi,i))\\
&=\P((Y_{\dd},\LL(\dd))\in A\times\{j\}|(Y_0,\LL(0))=(\xi,i))
\end{split}
\end{equation}
and
\begin{equation}\label{m3}
\P((Y_{(k+1)\dd},\LL((k+1)\dd))\in
A\times\{i\}|\F_{k\dd})=\P((Y_{(k+1)\dd},\LL((k+1)\dd))\in
A\times\{i\}|(Y_{k\dd},\LL(k\dd)))
\end{equation}
for any $A\in\mathscr{B}(\C)$ and $(\xi,i)\in\C\times\S$. }
\end{lem}

\begin{proof}
We shall verify \eqref{m2} and \eqref{m3} one-by-one. To begin, we
show that
  \eqref{m2} holds.
It is easy to see from \eqref{w2} that
\begin{equation}\label{m1}
Y_{k\dd}(i\dd)=Y((k+i)\dd),~~~~i=-M,\cdots,-1.
\end{equation}
Observe from \eqref{a1} and \eqref{m1} that
\begin{equation}\label{ww2}
\begin{aligned}
Y_\delta(\theta)&=Y((1+i)\delta)+\ff{\theta-i\delta}{\delta}\{
Y((2+i)\delta)-
Y((1+i)\delta)\}\\
&=\left\{\begin{array}{lll}
 Y(0)+\ff{\theta+\delta}{\delta}\{
Y(\delta)-
Y(0)\},~~~~~~~~~~~~~~~~~~~~~~~~~~~~~~\theta\in[-\dd,0]\\
 Y((1+i)\dd)+\ff{\theta-i\dd}{\dd}\{
Y((2+i)\dd)-
Y((1+i)\dd)\},~~~~\theta\in[i\dd,(i+1)\dd], ~~~i\neq-1\\
\end{array}\right.\\
&=\left\{\begin{array}{lll}
 Y(0)+\ff{\theta+\delta}{\delta}\{
b(Y_0,\LL(0))\dd+\si(Y_0,\LL(0))  W(\dd)\},~~~~~~\theta\in[-\dd,0],\\
 Y((1+i)\dd)+\ff{\theta-i\dd}{\dd}\{
Y((2+i)\dd)-
Y((1+i)\dd)\},~~~~\theta\in[i\dd,(i+1)\dd], ~~~i\neq-1\\
\end{array}\right.\\
\end{aligned}
\end{equation}
and that
\begin{equation}\label{m13}
\begin{aligned}
Y_{(k+1)\dd}(\theta)&=Y((k+1+i)\dd)+\ff{\theta-i\dd}{\dd}\{
Y((k+2+i)\dd)- Y((k+1+i)\dd)\}\\
&=\left\{\begin{array}{lll} Y_{k\dd}(0)\!+\!\ff{\theta\!+\!\dd}{\dd}\{
Y((k\!+\!1)\dd)\!-\! Y_{k\dd}(0)\},
~~~~~~~~~~~~~~~~~~~~~ \theta\in[-\dd,0]\\
Y_{k\dd}((1\!+\!i)\dd)\!+\!\ff{\theta\!-\!i\dd}{\dd}\{
Y_{k\dd}((2\!+\!i)\dd)\!-\! Y_{k\dd}((1\!+\!i)\dd)\},~~ \theta\in[i\dd,(i\!+\!1)\dd], ~ i\neq-1\\
\end{array}\right.\\
&=\left\{\begin{array}{lll} Y_{k\dd}(0)+\ff{\theta\!+\!\dd}{\dd}\{
b(Y_{k\dd},\LL(k\dd))\dd\\
~~~~~~~~~~~ +\si(Y_{k\dd},\LL(k\dd))
(W((k\!+\!1)\dd)\!-\!W(k\dd))\}, ~~\theta\in[-\dd,0]\\
Y_{k\dd}((1\!+\!i)\dd)\!+\!\ff{\theta\!-\!i\dd}{\dd}\{ Y_{k\dd}((2\!+\!i)\dd)\!-\!
Y_{k\dd}((1\!+\!i)\dd)\},~ \theta\in[i\dd,(i\!+\!1)\dd], ~ i\neq-1.
\end{array}\right.\\
\end{aligned}
\end{equation}
Thus, comparing \eqref{w2} with \eqref{m13} and noting that $
W((k+1)\dd)-W(k\dd)$ and $  W(\dd)$ are identical in distribution,
we infer that $(Y_{(k+1)\dd}, \LL((k+1)\dd))$  and $(Y_\dd,
\LL(\dd))$ are equal in distribution given
$(Y_{k\dd},\LL(k\dd))=(\xi, i)$ and $(Y_0,\LL(0))=(\xi, i)$,
respectively. Therefore, \eqref{m2} holds immediately.

Next, we demonstrate that   \eqref{m3} is fulfilled.  Set
\begin{equation*}
\chi^{\xi, j}_{(k+1)\dd}(\theta):=
\begin{cases}
\xi(0)+\frac{\theta+\dd}{\dd}\{b(\xi,j)\dd+\si(\xi,j)
(W((k+1)\dd)-W(k\dd))\}, ~~~~\theta\in[-\dd,0]\\
\xi((1+i)\dd)+\frac{\theta-i\dd}{\dd}[\xi((2+i)\dd)-\xi((1+i)\dd)],~~~~~~~~~~~~~~~~~
 \theta\in[i\dd,(i+1)\dd], ~ i\neq-1.
\end{cases}
\end{equation*}
and  $\Lambda_{k+1}^{j, \dd}:=j+\Lambda((k+1)\dd)-\Lambda(k\dd)$.
Thus, it is easy to see that
\begin{equation}\label{m4}
\LL((k+1)\dd)=\Lambda_{k+1}^{\LL(k\dd), \dd}~~~~~\mbox{ and }~~~~~~
Y_{(k+1)\dd}=\chi^{  Y_{k\dd}, \LL(k\dd)}_{(k+1)\dd}.
\end{equation}
For any $0\le s\le t$, let $\mathcal {G}_{t,s}=\si(W(u)-W(s),s\le
u\le t)\vee\mathcal {N}$. Plainly, $\mathcal {G}_{(k+1)\dd,k\dd}$ is
independent of $\F_{k\dd}$. Moreover,  $\chi_{(k+1)\dd}^{\xi,j}$
depends completely on the increment $W((k+1)\dd)-W(k\dd)$  so is
$\mathcal {G}_{(k+1)\dd,k\dd}$-measurable.  Hence,
$\chi_{(k+1)\dd}^{\xi,j}$ is independent of ${\cal F}_{k\dd}$.
Noting that $\chi_{(k+1)\dd}^{  Y_{k\dd}, \LL(k\dd)}$  and
$\Lambda_{(k+1)\dd}^{\LL(k\dd), \dd}$ are conditionally independent
given $(Y_{k\dd}, \LL(k\dd)).$ 
Therefore,
\begin{equation*}
\begin{split}
&\P((Y_{(k+1)\dd},\LL((k+1)\dd)\in A\times\{j\}|{\cal
F}_{k\dd})\\
&=\E(I_{A\times\{j\}}(\chi^{  Y_{k\dd},
\LL(k\dd)}_{(k+1)\dd},\Lambda_{k+1}^{\LL(k\dd)})|{\cal F}_{k\dd})\\
&=\E(I_A(\chi^{  Y_{k\dd}, \LL(k\dd)}_{(k+1)\dd})|{\cal
F}_{k\dd})\E(I_{\{j\}}(\Lambda_{k+1}^{\LL(k\dd)})|{\cal F}_{k\dd})\\
&=\E(I_A(\chi^{ \xi,
i}_{(k+1)\dd}))|_{\xi=Y_{k\dd},i=\LL(k\dd)}\E(I_{\{j\}}(\Lambda_{k+1}^{i}))|_{i=\LL(k\dd)}\\
&=\P(\chi^{ \xi, i}_{(k+1)\dd}\in
A)|_{\xi=Y_{k\dd},i=\LL(k\dd)}\P(\Lambda_{k+1}^{i}\in\{j\})|_{i=\LL(k\dd)}\\
&=\P((\chi^{ \xi, i}_{(k+1)\dd},\Lambda_{k+1}^{i})\in
A\times\{j\})|_{\xi=Y_{k\dd},i=\LL(k\dd)}\\
&=\P((Y_{(k+1)\dd},\LL((k+1)\dd))\in
A\times\{i\}|(Y_{k\dd},\LL(k\dd))).
\end{split}
\end{equation*}
So \eqref{m3} holds, and then $(Y_{k\dd},\LL(k\dd))$ is a
homogeneous Markov chain.
\end{proof}

\begin{lem}\label{le5}
{\rm Under the assumptions of Theorem \ref{th5},
  there exist $\dd_0\in(0,1)$ and $\aa>0$ such that
\begin{equation}\label{q1}
\E\|Y_{t_\dd}(\xi,i)-Y_{t_\dd}(\eta,i)\|_\8^2\le c\,\e^{-\aa
t}\|\xi-\eta\|^2_\8, ~~~~t\ge\tau, ~~\dd\in(0,\dd_0)
\end{equation}
for any $\xi,\eta\in\C$ and $i\in\S.$

}
\end{lem}

\begin{proof}
For fixed $\omega_2$, we focus on the following SDE
\begin{equation*}
\d Y^{\omega_2}(t) =b(Y_{t_\dd}^{\omega_2},\LL^{\omega_2}(t_\dd))\d
t+\si(Y_{t_\dd}^{\omega_2},\LL^{\omega_2}(t_\dd))\d \omega_1(t)
\end{equation*}
with the initial value $Y^{\omega_2}(\theta)=\xi(\theta),
\theta\in[-\tau,0],$ and $\LL^{\omega_2}(0)=i\in\S.$ Let
$\Upsilon^{\omega_2}(t)$ be defined as in \eqref{a0}. By \eqref{a1},
it is easy to see that
\begin{equation}\label{a2}
\E_{\P_1}|\Upsilon^{\omega_2}(t)-\Upsilon^{\omega_2}(t_\dd)|^2
\le(L_0+L)\dd\Big\{\E_{\P_1}|\Upsilon^{\omega_2}(t_\dd)|^2+\int_{-\tau}^0\E_{\P_1}|\Upsilon^{\omega_2}_{t_\dd}(\theta)|^2v(\d\theta)\Big\}.
\end{equation}
 Following the procedure to derive
\eqref{z1}, we obtain from ({\bf H1}), \eqref{a7} and \eqref{a2}
that
\begin{equation}\label{a4}
\begin{split}
&\e^{-\int_0^t\aa_{\LL^{\omega_2}(s_\dd)}\d
s}\E_{\P_1}|\Upsilon^{\omega_2}(t)|^2\\&=|\Upsilon^{\omega_2}(0)|^2+\int_0^t\e^{-\int_0^s\aa_{\LL^{\omega_2}(r_\dd)}\d
r}\E_{\P_1}\Big\{-\aa_{\LL^{\omega_2}(s_\dd)}|\Upsilon^{\omega_2}(s)|^2\\
&\quad+2\<\Upsilon^{\omega_2}(s),b(Y_{s_\dd}^{\omega_2}(\xi,i_0),\LL^{\omega_2}(s_\dd))-b(Y_{s_\dd}^{\omega_2}(\eta,i_0),\LL^{\omega_2}(s_\dd))\>\\
&\quad+\|\si(Y_{s_\dd}^{\omega_2}(\xi,i),\LL^{\omega_2}(s_\dd))-\si(Y_{s_\dd}^{\omega_2}(\eta,i),\LL^{\omega_2}(s_\dd))\|_{\rm
HS}^2\Big\}\d
s\\
&\le|\Upsilon^{\omega_2}(0)|^2+\int_0^t\e^{-\int_0^s\aa_{\LL^{\omega_2}(r_\dd)}\d
r}\Big\{(\check{\aa}+L_0)\sqrt{\dd}\E_{\P_1}|\Upsilon^{\omega_2}(s_\dd)|^2
\\
&\quad+(L_0\sqrt{\dd}+\bb_{\LL^{\omega_2}(s_\dd)})\int_{-\tau}^0\E_{\P_1}|\Upsilon^{\omega_2}_{s_\dd}(\theta)|^2v(\d\theta)+\ff{1+2\check{\aa}}{\sqrt{\dd}}\E_{\P_1}|\Upsilon^{\omega_2}(s)-\Upsilon^{\omega_2}(s_\dd)|^2\Big\}\d
s\\
&\le|\Upsilon^{\omega_2}(0)|^2+\nu\sqrt{\dd}\int_0^t\e^{-\int_0^s\aa_{\LL^{\omega_2}(r_\dd)}\d
r}\E_{\P_1}|\Upsilon^{\omega_2}(s_\dd)|^2\d s +\Psi^{\omega_2}(t),
\end{split}
\end{equation}
where $\nu:=\check{\aa}+L_0+(1+2\check{\aa})(L_0+L)$ and
\begin{equation*}
\Psi^{\omega_2}(t):=\int_0^t(\nu\sqrt{\dd}+\bb_{\LL^{\omega_2}(s_\dd)})\e^{-\int_0^s\aa_{\LL^{\omega_2}(r_\dd)}\d
r}\int_{-\tau}^0\E_{\P_1}|\Upsilon^{\omega_2}_{s_\dd}(\theta)|^2v(\d\theta)\d
s.
\end{equation*}
By virtue of \eqref{w2}, we deduce that
\begin{equation}\label{a3}
\begin{split}
&
\Psi^{\omega_2}(t)=\int_{-\tau}^0\int_0^t(\nu\sqrt{\dd}+\bb_{\LL^{\omega_2}(s_\dd)})\e^{-\int_0^s\aa_{\LL^{\omega_2}(r_\dd)}\d
r}\E_{\P_1}|\Upsilon^{\omega_2}_{s_\dd}(\theta)|^2\d
 sv(\d\theta)\\
 &\le2\sum_{i=-N}^{-1}\int_{i\dd}^{(i+1)\dd}\int_0^t(\nu\sqrt{\dd}+\bb_{\LL^{\omega_2}(s_\dd)})\e^{-\int_0^s\aa_{\LL^{\omega_2}(r_\dd)}\d
 r}\E_{\P_1}|\Upsilon^{\omega_2}(s_\dd+i\dd)|^2\d
 sv(\d\theta)\\
 &\quad+2\sum_{i=-N}^{-1}\int_{i\dd}^{(i+1)\dd}\int_0^t(\nu\sqrt{\dd}+\bb_{\LL^{\omega_2}(s_\dd)})\e^{-\int_0^s\aa_{\LL^{\omega_2}(r_\dd)}\d
 r}\E_{\P_1}|\Upsilon^{\omega_2}(s_\dd+(i+1)\dd)|^2\d
 sv(\d\theta)\\
 &\le c\,\|\xi-\eta\|_\8^2+\int_0^t\Theta_\nu^{\omega_2}(s)\e^{-\int_0^{s}\aa_{\LL^{\omega_2}(r_\dd)}\d
 r}\E_{\P_1}|\Upsilon^{\omega_2}(s_\dd)|^2\d s,
\end{split}
\end{equation}
where
\begin{equation}\label{q5}
\Theta^{\omega_2}_\nu(t):=2\e^{-\hat{\alpha}\tau}\sum_{i=-N}^{-1}\int_{i\dd}^{(i+1)\dd}\{2\nu\sqrt{\dd}
+\bb_{\LL^{\omega_2}(t_\dd-i\dd)}+\bb_{\LL^{\omega_2}(t_\dd-(i+1)\dd)}\}v(\d\theta).
\end{equation}
Inserting \eqref{a3} into \eqref{a4}, we arrive at
\begin{equation*}
\begin{split}
\e^{-\int_0^t\aa_{\LL^{\omega_2}(s_\dd)}\d
s}\E_{\P_1}|\Upsilon^{\omega_2}(t)|^2\le
c\,\|\xi-\eta\|_\8^2+\int_0^t(\nu\sqrt{\dd}+\Theta_\nu^{\omega_2}(s))\e^{-\int_0^s\aa_{\LL^{\omega_2}(r_\dd)}\d
r}\E_{\P_1}|\Upsilon^{\omega_2}(s_\dd)|^2\d s.
\end{split}
\end{equation*}
This, together with the fact that
\begin{equation}\label{a00}
\begin{split}
\Pi^{\omega_2}(t):=\e^{-\int_0^t\aa_{\LL^{\omega_2}(s_\dd)}\d
s}\sup_{t-\dd\le s\le
t}\E_{\P_1}|\Upsilon^{\omega_2}(s)|^2\le\e^{-\hat a
\dd}\sup_{t-\dd\le s\le
t}\Big(\e^{-\int_0^s\aa_{\LL^{\omega_2}(r_\dd)}\d
r}\E_{\P_1}|\Upsilon^{\omega_2}(s)|^2\Big),
\end{split}
\end{equation}
implies that
\begin{equation*}
\Pi^{\omega_2}(t)\le c\,\|\xi-\eta\|_\8^2+\e^{-\hat a
\dd}\int_0^t(\nu\sqrt{\dd}+\Theta_\nu^{\omega_2}(s))\Pi^{\omega_2}(s)\d
s.
\end{equation*}
Thus, an application of Gronwall inequality leads to
\begin{equation}\label{a5}
\Pi^{\omega_2}(t)\le c\,\|\xi-\eta\|_\8^2\e^{\e^{-\hat a
\dd}\int_0^t(\nu\sqrt{\dd}+\Theta_\nu^{\omega_2}(s))\d s}.
\end{equation}
Furthermore, observe that
\begin{equation}\label{a6}
\begin{split}
\int_0^t\Theta_\nu^{\omega_2}(s)\d
s
&=2\e^{-\hat{\alpha}\tau}\sum_{i=-N}^{-1}\int_{i\dd}^{(i+1)\dd}\int_{-i\dd}^{t-i\dd}\{2\nu\sqrt{\dd}
+\bb_{\LL^{\omega_2}(s_\dd)}\}\d sv(\d\theta)\\
&\quad+2\e^{-\hat{\alpha}\tau}\sum_{i=-N}^{-1}\int_{i\dd}^{(i+1)\dd}\int_{-(i+1)\dd}^{t-(i+1)\dd}\bb_{\LL^{\omega_2}(s_\dd)}\d sv(\d\theta)\\
&\le c+4\e^{-\hat{\alpha}\tau}\int_0^t\{\nu\sqrt{\dd}
+\bb_{\LL^{\omega_2}(s_\dd)}\}\d s.
\end{split}
\end{equation}
Hence, we infer from \eqref{a5} and \eqref{a6} that
\begin{equation*}
\begin{split}
\E|Y(t;\xi,i)-Y(t;\eta,i)|^2\le
c\,\|\xi-\eta\|_\8^2\E\,\e^{\e^{-\hat a \dd}(1+4\e^{-\hat
a\tau})\nu\sqrt{\dd} t+\int_0^t\{\aa_{\LL(s_\dd)}+4\e^{-\hat
a(\tau+\dd)}\}\bb_{\LL(s_\dd)}\d s}.
\end{split}
\end{equation*}
By applying Lemma \ref{exp} and combining with $\eta_3>0$, there
exist $\dd_0\in(0,1)$ and $\aa>0$ such that
\begin{equation}\label{q2}
\E|Y(t;\xi,i)-Y(t;\eta,i)|^2\le c\,\e^{-\aa t}\|\xi-\eta\|^2_\8,
~~~~t\ge\tau, ~~\dd\in(0,\dd_0).
\end{equation}
With ({\bf H2}) and \eqref{q2} in hand, \eqref{q1} can be obtained
by a standard procedure.
\end{proof}

\begin{lem}\label{le6}
{\rm Let the assumptions of Lemma \ref{le5} hold. Then, there exist
some $\dd_0\in(0,1) $ such that
\begin{equation}\label{q4}
\E\|Y_{t_\dd}(\xi,i)\|_\8^2\le c\, (1+\|\xi\|^2_\8), ~~~~t\ge\tau,
~~\dd\in(0,\dd_0)
\end{equation}
for any $(\xi,i)\in\C\times\S.$

}
\end{lem}

\begin{proof}
Mimicking the procedure to derive \eqref{b11}, we have
\begin{equation}\label{e1}
\begin{split}
&\e^{-\int_0^t(\gg+\aa_{\LL^{\omega_2}(s_\dd)})\d
s}\E_{\P_1}|Y^{\omega_2}(t)|^2\\&=|\xi(0)|^2+\int_0^t\e^{-\int_0^s(\gg+\aa_{\LL^{\omega_2}(r_\dd)})\d
r}\E_{\P_1}\Big\{-(\gg+\aa_{\LL^{\omega_2}(s_\dd)})|Y^{\omega_2}(s)|^2\\
&\quad+2\<Y^{\omega_2}(s),b(Y_{s_\dd}^{\omega_2},\LL^{\omega_2}(s_\dd))\>+\|\si(Y_{s_\dd}^{\omega_2},\LL^{\omega_2}(s_\dd))\|_{\rm
HS}^2\Big\}\d s,
\end{split}
\end{equation}
where $\gg>0$ is introduced in \eqref{eq60}. By ({\bf H2}) and
\eqref{a7}, it follows that
\begin{equation}\label{e2}
\begin{split}
\E_{\P_1}|Y^{\omega_2}(t)-Y^{\omega_2}(t_\dd)|^2&=\E_{\P_1}|b(Y_{t_\dd}^{\omega_2},\LL^{\omega_2}(t_\dd))|^2\dd^2+\E_{\P_1}\|\si(Y_{t_\dd}^{\omega_2},\LL^{\omega_2}(t_\dd))\|^2_{\rm
HS}\dd\\
&\le
c+2(L_0+L)\dd\Big\{\E_{\P_1}|Y^{\omega_2}(t_\dd)|^2+\int_{-\tau}^0\E_{\P_1}|Y_{t_\dd}^{\omega_2}(\theta)|^2v(\d\theta)\Big\}.
\end{split}
\end{equation}
Then, taking \eqref{e1} and \eqref{e2} into consideration, we deduce
that
\begin{equation}\label{q0}
\begin{split}
&\e^{-\int_0^t(\gg+\aa_{\LL^{\omega_2}(s_\dd)})\d
s}\E_{\P_1}|Y^{\omega_2}(t)|^2\\
&\le|\xi(0)|^2+\int_0^t\e^{-\int_0^s(\gg+\aa_{\LL^{\omega_2}(r_\dd)})\d
r}\Big\{c+\ff{1+2(\gg+\check{\aa})}{\sqrt{\dd}}\E_{\P_1}|Y^{\omega_2}(s)-Y^{\omega_2}(s_\dd)|^2\\
&\quad+(\gg+\check{\aa})\sqrt{\dd}\E_{\P_1}|Y^{\omega_2}(s_\dd)|^2+(\gg+\bb_{\LL^{\omega_2}(s_\dd)})\int_{-\tau}^0\E_{\P_1}|Y^{\omega_2}_{s_\dd}(\theta)|^2v(\d\theta)\\
&\quad+\sqrt{\dd}\E_{\P_1}|b(Y_{s_\dd}^{\omega_2},\LL^{\omega_2}(s_\dd))|^2\Big\}\d
s\\
&\le|\xi(0)|^2+\int_0^t\e^{-\int_0^s(\gg+\aa_{\LL^{\omega_2}(r_\dd)})\d
r}\Big\{c+\rho\sqrt{\dd}\,\E_{\P_1}|Y^{\omega_2}(s_\dd)|^2\Big\}\d
s+\Psi^{\omega_2}_0(t),
\end{split}
\end{equation}
where $\rho:=4(1+\gg+\check{a})(L_0+L)$ and
\begin{equation*}
\Psi^{\omega_2}_0(t):=\int_0^t(\rho\sqrt{\dd}+\bb_{\LL^{\omega_2}(s_\dd)})\e^{-\int_0^s(\gg+\aa_{\LL^{\omega_2}(r_\dd)})\d
r}\int_{-\tau}^0\E_{\P_1}|Y^{\omega_2}_{s_\dd}(\theta)|^2v(\d\theta)\d
s.
\end{equation*}
 Following the argument to
deduce \eqref{a3}, we find that
\begin{equation}\label{q6}
\Psi^{\omega_2}_0(t)\le
c\,\|\xi\|_\8^2+\int_0^t\Theta^{\omega_2}_\rho(s)\e^{-\int_0^{s}(\gg+\aa_{\LL^{\omega_2}(r_\dd)})\d
 r}\E_{\P_1}|Y^{\omega_2}(s_\dd)|^2\d s,
\end{equation}
where $\Theta^{\omega_2}_\rho>0$ is defined as in \eqref{q5}.
Substituting \eqref{q6} into \eqref{q0} leads to
\begin{equation*}
\begin{split}
&\e^{-\int_0^t(\gg+\aa_{\LL^{\omega_2}(s_\dd)})\d
s}\E_{\P_1}|Y^{\omega_2}(t)|^2\\
&\le
c\,\|\xi\|^2_\8+\int_0^t\e^{-\int_0^s(\gg+\aa_{\LL^{\omega_2}(r_\dd)})\d
r}\Big\{c+(\rho\sqrt{\dd}+\Theta^{\omega_2}_\rho(s))\E_{\P_1}|Y^{\omega_2}(s_\dd)|^2\Big\}\d
s.
\end{split}
\end{equation*}
This, applying the Gronwall inequality and  utilizing \eqref{a00}
with $\Upsilon^{\omega_2}$ being replaced by $Y^{\omega_2}$ enables
us to obtain that
\begin{equation*}
\begin{split}
\Pi^{\omega_2}(t)
&\le
c\,\|\xi\|^2_\8+c\int_0^t\e^{-\int_0^s(\gg+\aa_{\LL^{\omega_2}(r_\dd)})\d
r}\d s+\e^{-\hat a
\dd}\int_0^t\Big\{c\,\|\xi\|^2_\8+c\int_0^s\e^{-\int_0^u(\gg+\aa_{\LL^{\omega_2}(r_\dd)})\d
r}\d u\Big\}\\
&\quad\times(\rho\sqrt{\dd}+\Theta^{\omega_2}(s))\e^{\int_s^t\e^{-\hat
a \dd}(\rho\sqrt{\dd}+\Theta^{\omega_2}(r))\d r}\d s.
\end{split}
\end{equation*}
Subsequently, the desired assertion follows from Fubini's theorem
and Lemma \ref{exp} and by taking $\gg,\dd\in(0,1)$ sufficiently
small.
\end{proof}

So far, the proof of Theorem \ref{th5} can be available.
\begin{proof}
With the help of Lemmas \ref{Markov}-\ref{le6}, we can finish the
proof by following the argument of Theorem \ref{th4}.
\end{proof}


\begin{thebibliography}{17}
{\small

\setlength{\baselineskip}{0.14in}
\parskip=0pt


\bibitem{AWM} Appleby, J.A.D.,
Wu, H.,  Mao, X, On the almost sure running maxima of solutions of
affine neutral stochastic functional diffenrential equations,
arXiv:1310.2349v1.


\bibitem{AHM} Arriojas, M.,  Hu, Y.,  Mohammed, S.-E.,   Pap,  G., A delayed Black
and Scholes formula, {\it Stoch. Anal. Appl.}, {\bf25} (2007),
471--492.

\bibitem{BSY} Bao, J., Shao, J., Yuan, C., Approximation of invariant
measures for regime-switching diffusions,
 {\it Potential Anal.}, {\bf44} (2016),   707--727.


\bibitem{B10} Bardet, J.-B., Gu\'{e}rin, H., Malrieu, F., Long time behavior of diffusions with Markov
switching,
 {\it ALEA Lat. Am. J. Probab. Math. Stat.}, {\bf7} (2010), 151--170.

\bibitem{BL}Benaim, M., Le Borgne, S.,
Malrieu, F., Zitt, P.-A., Quantitative ergodicity for some switched
dynamical systems, {\it Electron. Commun. Probab.}, {\bf17} (2012),
  14 pp.


\bibitem{BR}Buckwar, E., Riedler, M., An exact stochastic hybrid model of
excitable membranes including spatio-temporal evolution, {\it J.
Math. Biol.}, {\bf63} (2011), 1051--1093.


\bibitem{BS} Butkovsky, O.,  Scheutzow£¬ M., Invariant measures for stochastic functional differential
equations,  arXiv:1703.05120.


\bibitem{CC} Caraballo, T., Chueshov, I.~D., Marin-Rubio, P.,
Real, J., Existence and asymptotic behaviour for stochastic heat
equations with multiplicative noise in materials with memory, {\it
Discrete Contin. Dyn. Syst.}, {\bf18} (2007),  253--270.


\bibitem{CH} Cloez, B., Hairer, M., Exponential ergodicity for Markov processes with random switching, {\it Bernoulli}, {\bf21} (2015),  505--536.



\bibitem{DY}
 de Saporta, B., Yao, J., Tail of a linear diffusion with Markov switching, {\it  Ann. Appl. Probab.}, {\bf15} (2005),   992--1018.


 \bibitem{ESV} Es-Sarhir, A., Scheutzow, M., van Gaans, O., Invariant measures for stochastic functional differential equations with
 superlinear drift term, {\it Differential Integral Equations}, {\bf23} (2010),   189--200.




\bibitem{HMS} Hairer, M., Mattingly, J.~C., Scheutzow, M., Asymptotic coupling
and a general form of Harris' theorem with applications to
stochastic delay equations,
 {\it Probab. Theory Related Fields}, {\bf149} (2011),  223--259.



\bibitem{HL}Hale, J.~K., Lunel, S.~M., \emph{Introduction to Functional
Differential Equations}, Applied Mathematical Sciences.
Springer-Verlag, New York, 1993.






\bibitem{KMS}Katsoulakis, M.~A., Majda, A.~J., Sopasakis, A., Intermittency,
metastability and coarse graining for coupled
deterministic-stochastic lattice systems, {\it Nonlinearity },
{\bf19} (2006), 1021--1323.



\bibitem{KW}
 Kinnally, M.~S., Williams, R.~J., On existence and uniqueness of stationary distributions for stochastic
 delay differential equations with positivity constraints,
 {\it Electron. J. Probab.}, {\bf15} (2010),   409--451.





\bibitem{LS}Li, J., Shao, J., Algebraic stability of non-homogeneous regime-switching diffusion
processes,  arXiv:1509.04768v2.

\bibitem{LMY}Lygeros, J., Mao, X., Yuan, C., Stochastic hybrid
delay population dynamics. Hybrid systems: computation and control,
436--450, Lecture Notes in Comput. Sci., 3927, Springer, Berlin,
2006.



\bibitem{MT}Majda, A.~J., Tong, X., Geometric ergodicity for piecewise
contracting processes with applications for tropical stochastic
lattice models,
  {\it Comm. Pure Appl. Math.}, {\bf69} (2016),  1110--1153.


\bibitem{MS}Mao, X., Sabanis, S., Delay geometric Brownian motion
in financial option valuation, {\it Stochastics}, {\bf85} (2013),
  295--320.


\bibitem{MTY}Mao, X., Truman, A., Yuan, C., Euler-Maruyama
approximations in mean-reverting stochastic volatility model under
regime-switching, {\it J. Appl. Math. Stoch. Anal.}, 2006, Art. ID
80967, 20 pp.


\bibitem{Mao}Mao, X.,  Delay population dynamics and environmental noise, {\it Stoch.
Dyn.},{\bf 5} (2005),  149--162.
\bibitem{M84}
Mohammed, S-E. A., Stochastic Functional Differential Equations,
Pitman, Boston, 1984.


\bibitem{PP}Pinsky, M., Pinsky, R., Transience recurrence and central limit theorem behavior for diffusions
in random temporal enrivoments, {\it Ann. Probab.}, {\bf21} (1993),
433--452.



\bibitem{PS} Pinsky, M., Scheutzow, M., Some remarks and examples
concerning the transience and recurrence of random diffusions, {\it
Ann. Inst. Henri. Poincar\'{e}}, {\bf28} (1992), 519--536.

\bibitem{RRV} Rei\ss, M., Riedle, M., van Gaans, O., Delay differential equations
driven by L\'{e}vy processes: stationarity and Feller properties,
  {\it Stochastic Process. Appl.}, {\bf116} (2006),   1409--1432.

\bibitem{Shao} Shao, J., Ergodicity of regime-switching diffusions in Wasserstein distances, {\it Stochastic Process. Appl.},
{\bf 125} (2015),  739--758.

\bibitem{Sh15}Shao, J., Criteria for transience and recurrence of regime-switching diffusion processes,    {\it Electron. J. Probab.}, {\bf 20} (2015), no. 63, 1--15.

\bibitem{SX}
 Shao, J., Xi, F., Strong ergodicity of the regime-switching diffusion processes, {\it Stochastic Process. Appl.}, {\bf 123} (2013),  3903--3918.


\bibitem{SYZ} Song, Q., Yin, G., Zhu, C., Optimal switching with
constraints and utility maximization of an indivisible market,
 {\it SIAM J. Control Optim.}, {\bf50} (2012),   629--651.

\bibitem{TM} Tong, X., Majda, A.~J.,
  Moment bounds and geometric ergodicity of diffusions with random switching and unbounded transition rates,
   {\it Res. Math. Sci.}, {\bf3} (2016), Paper No. 41, 33 pp.

\bibitem{VS} von Renesse, Max-K., Scheutzow, M., Existence and uniqueness of solutions of stochastic functional differential
equations,
 {\it Random Oper. Stoch. Equ.}, {\bf18} (2010),   267--284.


\bibitem{XZ} Xi, F., Zhu, C., On Feller and Strong Feller Properties and Exponential Ergodicity of
Regime-Switching Jump Diffusion Processes with Countable Regimes,
arXiv:1702.01048.

\bibitem{YM} Yuan, C., Mao, X., Asymptotic stability in
distribution of stochastic differential equations with Markovian
switching,
  {\it Stochastic Process. Appl.}, {\bf103} (2003),   277--291.

\bibitem{YMa}Yuan, C., Mao, X., Stationary distributions of Euler-Maruyama-type
stochastic difference equations with Markovian switching and their
convergence, {\it J. Difference Equ. Appl.}, {\bf11} (2005), 29--48.



\bibitem{YZM} Yuan, C., Zou, J., Mao, X., Stability in
distribution of stochastic differential delay equations with
Markovian switching, {\it Systems Control Lett.}, {\bf50} (2003),
  195--207.


\bibitem{ZY}Zhu, C., Yin, G., On competitive Lotka-Volterra model in random
environments, {\it J. Math. Anal. Appl.}, {\bf357} (2009), 154--170.






}
\end{thebibliography}
\end{document}